\documentclass[preprint,12pt]{elsarticle}

\usepackage{amssymb}
\usepackage{amsmath}
\usepackage{amsthm}
\usepackage{color}
\usepackage{mathrsfs}
\usepackage{arydshln}
\usepackage{url}

\theoremstyle{definition}
\newtheorem{Def}{Definition}[section]

\newtheorem{Eg}{Example}[section]

\theoremstyle{plain}
\newtheorem{Prop}[Def]{Proposition}
\newtheorem{Lem}[Def]{Lemma}
\newtheorem{Thm}[Def]{Theorem}
\newtheorem{Cor}[Def]{Corollary}

\journal{Stochastic Processes and their Applications}

\begin{document}

\begin{frontmatter}



\title{Regularization Effects of Time Integration on Gaussian Process Functionals}


\author[label1]{Takafumi Amaba} %% Author name
\ead{fmamaba@fukuoka-u.ac.jp}

\affiliation[label1]{
    organization={Department of Applied Mathematics, Fukuoka university},%Department and Organization
    addressline={8-19-1 Nanakuma, Jonan-ku}, 
    city={Fukuoka},
    postcode={814-0180}, 
    state={},
    country={Japan},
}

\author[label2]{Marie Kratz}
\ead{kratz@essec.edu}
\affiliation[label2]{
    organization={ESSEC Business School, CREAR},
    city={Cergy-Pontoise},
    country={France}
}

%% Abstract
\begin{abstract}
In this paper, we investigate the regularization effects, in the sense of Malliavin calculus, on functionals of Gaussian processes induced by time integration, focusing on their covariance functions. We study several examples of important covariance functions classes to verify whether they satisfy the sufficient conditions proposed for regularization. Additionally, we derive a weak implication for the smoothness of level-crossing functionals.
\end{abstract}


% \begin{highlights}
% \end{highlights}

%% Keywords

\end{frontmatter}



%%%%%%%%%%%%%%%%%%%%%%%%%%%%%%%%%%%%%%%%
\section{Introduction}
\label{Intro} %@@@@@@@@@@@@@@@@@@@@@@@@@ LABEL: Intro
%%%%%%%%%%%%%%%%%%%%%%%%%%%%%%%%%%%%%%%%

When studying a given function, calculating its derivative is a standard procedure in mathematics. Similarly, when studying a random variable and its distribution, it is natural to consider whether it is possible to apply the same approach of differentiation that we use for functions.
The Malliavin calculus, initiated by P.~Malliavin (\cite{Ma97}) and developed through studies by D.~W.~Stroock, J.~M.~Bismut, P.~A.~Meyer (\cite{Me84}), S.~Watanabe (\cite{watanabe-book, Wa93}) and D.~Nualart (\cite{Nualart-book06}), among others, provides a framework for discussing the smoothness of random variables as functions on a probability space. In this paper, we denote by $\mathbb{D}_{2}^{s}$ the Hilbert space consisting of the `$s$-times differentiable' Meyer–Watanabe distribution (which is denoted $\mathbb{D}^{2,s}$ in the book by Ikeda–Watanabe~\cite{IW89}).

It is important to note that, from the perspective of approximation theory for a given random variable (such as the Wiener–Poisson functional), identifying the `differentiability index' $s$ is essential. This index represents the largest value of $s$ such that the space $\mathbb{D}_{2}^{s}$ contains the random variable. Such considerations have been addressed in many works, including \cite{Ami92,NVives92b,AkAmOk13,Ama16,ALMS19,NYY22,GeGeLa13,GeGe04,GeHu07}, and others.

%Let us remark that, as shown in studies by Amine (\cite{Ami92}), Akahori–Amaba–Okuma (\cite{AkAmOk13}), Amaba (\cite{Ama16}), Amaba–Liu–Makhlouf–Saidaoui (\cite{ALMS19}), Nishimura–Yasutomi–Yuasa (\cite{NYY22}), Geiss–Geiss–Laukkarinen (\cite{GeGeLa13}), Geiss–Geiss (\cite{GeGe04}), Geiss–Hujo (\cite{GeHu07}), and so on, it is very important from the point of view of the approximation theory of a given random variable (Wiener–Poisson functional) to identify the `differentiability index' $s$, namely how large $s$ is such that the space $\mathbb{D}_{2}^{s}$ contains the random variable.

For a given stochastic process $X = \{ X_{t} \}_{t \in \mathbf{R}}$ and a function $\Lambda (x)$, the smoothness of $\Lambda (X_{t})$ can be formulated within the framework of Malliavin calculus, by considering for which $s$ the random variable (or Meyer--Watanabe distribution) $\Lambda (X_{t})$ belongs to $\mathbb{D}_{2}^{s}$. Then (Bochner) integration of $\Lambda (X_{t})$ with respect to time $t$ makes it, of course, smoother with respect to the time direction. Apart from this, in some situations, the smoothness increases in the direction of the underlying probability space by taking integral with respect to time. This is a phenomenon we call a `regularization effect,' and this is the subject of this paper.

When $X_{t}$ is the solution to a stochastic differential equation with smooth coefficients, particularly when the diffusion coefficient is uniformly nondegenerate, it is known that the smoothness of $\int_{0}^{t} \Lambda (X_{s}) \,\mathrm{d}s$ is one order higher than that of $\Lambda (X_{t})$; see e.g. \cite{NVives92b,im:pv,Tu13,amaba:ryu}. 

% When $X_{t}$ is obtained as a solution to a stochastic differential equation whose coefficients are smooth, especially when the diffusion coefficient is uniformly nondegenerate, it is known that the smoothness of $\int_{0}^{t} \Lambda (X_{s}) \,\mathrm{d}s$ is one higher than that of $\Lambda (X_{t})$; see e.g. \cite{NVives92b,Tu13,amaba:ryu}. 
This provides an explanation from the perspective of Malliavin calculus that local time, which has the heuristic expression including `$\int_{0}^{t} \delta_{0} ( X_{s} ) \,\mathrm{d}s$,' is determined as a random variable rather than a Meyer–Watanabe distribution. One of the key ideas behind this is the elliptic regularity theorem for the generators associated with the solution to the stochastic differential equation, which further proves this regularization effect by formulating Ito's formula for distributions.

In this paper, we show that such a regularization phenomenon also occurs when $X_{t}$ is a differentiable stationary Gaussian process. Since $X_{t}$ is differentiable with respect to $t$ in our case, let us consider an integrable function of the form $\Lambda ( X_{t} , \dot{X}_{t} )$ in a slightly more general way. Interestingly, the degree of regularization differs from that of the diffusion case above. Furthermore, the key to the mechanism of regularization is the Gauss hypergeometric theorem, which is also completely different from the diffusion case.

Now let us outline our main result, which will be refined with the precise conditions in Theorem~\ref{Thm:raise}.
% we are in a position to state (a rough version of) our main result. 

Let $r$ be the covariance function of the Gaussian stationary process $X$.
We denote by $\mathscr{S}^{\prime}(\mathbf{R}^{k})$ the space of all tempered distributions on $\mathbf{R}^{k}$.

\vspace{1ex}
%\begin{Thm}[See Theorem~\ref{Thm:raise} for the precise statement] %%%%%%%%%%%%%%%%%%%%%%%%%%%
%\label{Thm:main} %@@@@@@@@@@@@@@@@@@@@@@@@@ LABEL: Thm:main
{\bf Theorem} (See Theorem~\ref{Thm:raise} for the precise statement).

{\it 
Let $\alpha \in \mathbf{R}$ be arbitrary.
There exists a constant $c=c(r) > 0$ such that the following holds true:
\begin{itemize}
\item[{\rm (1)}]
%Under (A1)(roughly speaking, 
If $r$ is twice continuously differentiable (under Condition (A1), to be precise), for any
$\Lambda \in \mathscr{S}^{\prime} (\mathbf{R})$ with $\Lambda ( X_{0} ) \in \mathbb{D}_{2}^{\alpha}$,
it holds that
$$
\Vert \int_{0}^{1} \Lambda ( X_{t} ) \,\mathrm{d}t \Vert_{2,\alpha + \frac{1}{2}} \leqslant c\Vert \Lambda ( X_{0} )\Vert_{2,\alpha} .
$$
\item[{\rm (2)}] 
If $r$ is fourth differentiable (under (A2), to be precise), for any
$\Lambda \in \mathscr{S}^{\prime} (\mathbf{R}^{2})$ with $\Lambda ( X_{0}, \dot{X}_{0} ) \in \mathbb{D}_{2}^{\alpha}$,
it holds that
$$
\Vert \int_{0}^{1} \Lambda ( X_{t}, \dot{X}_{t} ) \,\mathrm{d}t \Vert_{2,\alpha + \frac{1}{2}} \leqslant c \Vert \Lambda ( X_{0}, \dot{X}_{0} ) \Vert_{2,\alpha}.
$$
\end{itemize}
}
%
%\end{Thm} %%%%%%%%%%%%%%%%%%%%%%%%%%%%%

We can rephrase this result as: {\it For smooth Gaussian processes, the Bochner integral $\int_{0}^{1}\mathrm{d}t$ regularizes $\Lambda ( X_{t}, \dot{X}_{t} )$ by at least $1/2$},
which should be compared with the result \cite[Corollary~1.6]{amaba:ryu}.
For the degree of regularization, this is the best possible result in the sense that the above holds for all $\Lambda$.

An interesting question arises regarding the conclusions that may be drawn for the `level crossing functional.' The current theorem tells us that it belongs to $\displaystyle \cap_{s<0} \mathbb{D}_{2}^{s}$.
However, this result is slightly weaker than the conclusion that the level crossing functional belongs to $L_{2}$ ($= \mathbb{D}_{2}^{0}$), as suggested by the Geman condition (see \cite{geman,probasurveyMK}, and discussed later in Corollary~\ref{Cor:Ocup_Sob}). This discrepancy indicates that a stronger form of regularization may be acting on the level crossing functional, which cannot be fully explained by the present general theory. A more detailed analysis of this aspect, using another approach, will be provided in our future work.

{\it Structure of the paper.} In Section~\ref{Sec:fmwork}, we present the mathematical framework of our study, as well as the conditions on the covariance function of the Gaussian process under which our main regularization result will hold. The main result and its application to level-crossing functionals are provided in Section~\ref{Sec:Results}, while their proofs and other technical results, of independent interest, are deferred to Section~\ref{Sec:Proof}. Important classes of covariance functions
(see e.g., \cite[Chapter~4]{WR06}),
like Mat\'ern classes, are studied in Section~\ref{Sec:Examples} to verify whether they satisfy the proposed conditions under which the main result holds, to make it also relevant in practice.

%%%%%%%%%%%%%%%%%%%%%%%%%%%%%%%%%%%%%%%%
\section{Mathematical Framework}
\label{Sec:fmwork}
%%%%%%%%%%%%%%%%%%%%%%%%%%%%%%%%%%%%%%%%

\subsection{Notations}
\label{ssec:notation}
%To make this paper to be self-contained, we start with the definition of Gaussian processes.

Let $(\Omega , \mathcal{F}, \mathbb{P} )$ be a probability space.
Let $\mathbf{Z}^{+}$ denote the set of non-negative integers, $\mathbf{N}$ the set of positive integers, and $\mathbf{R}^{+}$ the set of non-negative real numbers.

Recall that, if $X = \{ X_{t} \}_{t \geqslant 0}$ is a Gaussian process on $\mathbf{R}$,
then the mapping
$ t \in [0,+\infty ) \mapsto X_{t} \in L_{2}( \Omega ) $
is injective.
Indeed, let $s,t \in \mathbf{R}$ be such that $s \neq t$.
If $X_{s} = X_{t}$, then
$\sigma_{11} = \mathrm{Var}(X_{s})$,
$\sigma_{12} = \mathrm{Cov}(X_{s}, X_{t})$,
$\sigma_{21} = \mathrm{Cov}(X_{t}, X_{s})$
and
$\sigma_{22} = \mathrm{Var}(X_{t})$
are equal, say $\sigma$.
Then
$$
\Sigma
=
\left(\begin{array}{cc}
\sigma & \sigma \\
\sigma & \sigma
\end{array}\right)
$$
is not invertible, which contradicts the definition of a Gaussian process.
% to Definition~\ref{Def:GP}--(1a).
Hence, it must be that $X_{t} \neq X_{s}$ in $L_{2}( \Omega )$.
%\end{Rm} %%%%%%%%%%%%%%%%%%%%%%%%%%%%%%%

Let $X = \{ X_{t} \}_{t \in \mathbf{R}}$ be a centered and stationary Gaussian process with $\mathrm{Var} ( X_{0} ) = 1$.
Then its {\it covariance function} $r : \mathbf{R} \to \mathbf{R}$ is the even function defined by
\begin{equation*}
r (t) = \mathrm{Cov} ( X_{t+h}, X_{h} ) = \mathbb{E} [ X_{t+h} X_{h} ] \quad \text{for $t \in \mathbf{R}$}
\end{equation*}
(that does not depend on $h \in \mathbf{R}$ because of the stationarity),
satisfying $r(t)\leqslant r(0)=1$,
% It is clear that $r(0) = 1$ and $r(t)\leqslant(\mathrm{Var} (X_{t})\mathrm{Var} (X_{0}))^{1/2}= 1$by Cauchy--Schwarz's inequality, 
with equality that holds only when $X_{0} = X_{t}$, {\it i.e.} for $t=0$.
Thus, {\it $t=0$ is a strict global maximum point of the function $r$}.

Recall that the covariance function $r$ is {\it positive definite}
(in the sense of that
for any $n \in \mathbf{N}$, $t_{1}, t_{2}, \ldots , t_{n} \in \mathbf{R}$ and $c_{1}, c_{2}, \ldots , c_{n} \in \mathbf{C}$,
$\displaystyle 
\sum_{j,k=1}^{n} c_{j} \overline{c_{k}}r( t_{j} - t_{k} )
= \mathbb{E}[\vert \sum_{j=1}^{n} c_{j} X_{t_{j}} \vert^{2}] \geqslant 0$.
Then, by Bochner's theorem, there exists a unique probability measure $F$ on $\mathbf{R}$,
called the {\it spectral distribution}, such that
\begin{equation*}
r(t)=\int_{-\infty}^{\infty}\mathrm{e}^{i\,\lambda t}\,\mathrm{d}F (\lambda),\quad t \in \mathbf{R}.
\end{equation*}
Suppose that $F$ is absolutely continuous and let $F^{\prime}$ denote its Radon--Nikod\'{y}m derivative $\mathrm{d}F(\lambda)/\mathrm{d} \lambda$.
% of $\mathrm{d}F(\lambda)$ with respect to $\mathrm{d} \lambda$.
Since $r(\cdot)$ is real, we have
$ F^{\prime} (\lambda)= F^{\prime} (-\lambda) $ for almost every $\lambda \in \mathbf{R}$.
Then, assuming that
$ ( F^{\prime} )^{1/2} \in L_{1} (\mathbf{R})$,
the continuous function $b \in L_{2}(\mathbf{R})$ defined by
\begin{equation}
\label{eq:b}
b(x)=\frac{1}{ 2\pi }\int_{-\infty}^{\infty}F^{\prime}(\lambda )^{1/2}\mathrm{e}^{i\,x\lambda}\,\mathrm{d}\lambda ,\quad x \in \mathbf{R},
\end{equation}
is real-valued, and $r(\cdot)$ has the following representation (see \cite[p.~149]{Berman-book}):
\begin{equation}\label{Corr-X}
r(t) = \int_{-\infty}^{\infty} b(t+s) b(s) \,\mathrm{d}s, \quad t \in \mathbf{R}.
\end{equation}
In this point of view, the law of the Gaussian process $X$ coincides
with that of $\{ W(b_{t}) \}_{t \in \mathbf{R}}$,
where $b_{t} (x) := b(t+x)$ for $t,x \in \mathbf{R}$ and
\begin{equation}
\label{def-W} 
X_t=W (b_t)=\int_{-\infty}^{\infty} b(t+x) \,\mathrm{d} w(x),
\end{equation}
$\{ w(t) \}_{t \in \mathbf{R}}$ denoting a one-dimensional Brownian motion
(see \cite[p.~157]{Berman-book}).

Let $ \mathcal{F}^{\mathrm{d}w} $ be the $\mathbb{P}$-completion of the $\sigma$-field
$\sigma \{ w(t)-w(s) : - \infty < s < t < +\infty \}$.
Let $L_{p}( \Omega )$ denote the space of real-valued, $\mathcal{F}^{ \mathrm{d}w }$-measurable and $p$-th order integrable (with respect to $\mathbb{P}$) random variables $Z$
with norm $\Vert Z \Vert_{p} = \mathbb{E} [ \vert Z \vert^{p} ]^{1/p}$.

We have an orthogonal decomposition
$\displaystyle  L_{2} ( \Omega ) = \bigoplus_{n=0}^{\infty} \mathcal{H}_{n}$,
where $\mathcal{H}_{0} = \mathbf{R}$ and, for $n \geqslant 1$, $\mathcal{H}_{n}$
is the closure in $\displaystyle L_{2}(\Omega , \mathcal{F}^{ \mathrm{d}w }, \mathbb{P})$ of the subspace spanned by $\{H_{n} (W(h)) : h(x) \in L_2 ( \mathbf{R}, \mathrm{d}x )\}$,
and $H_{n}$ denotes the $n$-th Hermite polynomial defined as 
$$
\mathrm{e}^{ tx- \frac{t^2}{2} } = \sum_{n=0}^{\infty} H_n (x) \frac{t^n}{n!} \qquad \mbox{or} \qquad
H_n(x) = (-1)^{n} \mathrm{e}^{x^2/2} \frac{ \mathrm{d}^{n} }{ \mathrm{d} x^{n} } ( \mathrm{e}^{-x^2/2}).
$$ 
We can make use as well of the multiple Wiener–It\^o integral $I_n$ defined as in Major~\cite{major}, since we have
$$
H_{n} (W(h))=I_{n} ( h^{\otimes n} )\qquad
\text{for $h \in L_{2} (\mathbf{R})$ such that $\Vert h \Vert_{L_{2}(\mathbf{R})} = 1$.}
$$
The orthogonal projection $L_{2} (\Omega ) \to \mathcal{H}_{n}$ will be denoted by $J_{n}$.

For each $\alpha \in \mathbf{R}$ and $p \in (1,\infty )$, a Sobolev-type space $\mathbb{D}_{p}^{\alpha}$
is defined as the completion of
$\displaystyle \mathcal{P}:=\cup_{n=1}^{\infty} \cap_{m \geqslant n}\{f \in L_{2}(\Omega):J_{m} f = 0\}$
under the norm $\Vert \cdot  \Vert_{p,\alpha}$ defined by
$\displaystyle \Vert f \Vert_{p, \alpha} = \Vert (I-\mathcal{L})^{\alpha /2} f \Vert_{p} $
for $f \in \mathcal{P}$,
where $\mathcal{L}$ is the Ornstein-Uhlenbeck operator given by
\begin{equation*}
(I-\mathcal{L})^{\alpha /2}f = \sum_{k=0}^{\infty} (1+k)^{\alpha /2} J_{k}f, \quad f \in \mathcal{P}.
\end{equation*}
Note that $\mathbb{D}_{p}^{0} = L_{p}(\Omega )$
for $p \in (1, \infty)$, and
\begin{equation*} 
\Vert f \Vert_{2, \alpha}^{2} = \sum_{k=0}^{\infty} (1+k)^{\alpha}\Vert J_{k}f \Vert_{2}^{2}, \quad f \in \mathbb{D}_{2}^{\alpha}.
\end{equation*}
We further define
$\mathbb{D}_{p}^{\alpha -} := \cap_{\beta < \alpha} \mathbb{D}_{p}^{\beta}$ for $p \in (1, \infty )$
and $\alpha \in \mathbf{R}$,
\begin{equation*}
\mathbb{D}^{\infty} := \bigcap_{\alpha >0} \bigcap_{1 < p < \infty} \mathbb{D}_{p}^{\alpha} 
\quad \text{and} \quad 
\mathbb{D}^{-\infty} := \bigcup_{\alpha <0} \bigcup_{1 < p < \infty} \mathbb{D}_{p}^{\alpha}.
\end{equation*}
It is known that $\displaystyle (\mathbb{D}_{p}^{\alpha})^{\prime} = \mathbb{D}_{q}^{-\alpha}$
if and only if $1/p + 1/q = 1$ (where ``$\prime$" stands for the ``continuous dual")
for each $\alpha \in \mathbf{R}$, the space $\mathbb{D}^{\infty}$ is a complete countably-normed space
and $\mathbb{D}^{-\infty}$ is its dual named the space of {\it generalized Wiener functionals}.

%%%%%%%%%%%%%%%%%%%%%%%%%%%%%%%%%%%%%%%%
\subsection{Conditions and properties}
\label{ssec:Cdtion} %@@@@@@@@@@@@@@@@@@@@ LABEL: ssec:Cdtion
%%%%%%%%%%%%%%%%%%%%%%%%%%%%%%%%%%%%%%%%

In this section, we introduce several conditions under which we will be working in this study.

Let $b \in L^2(\mathbf{R})$ satisfying \eqref{Corr-X}, 
so that the differentiability of $r$ can be inferred from that of $b$.
Consider the following set of conditions on $b$ to ensure the twice differentiability of $r$
(see Proposition~\ref{r:diff'ble}).

%\begin{Def} %%%%%%%%%%%%%%%%%%%%%%%%%%%
%\label{Def:(A1)} %@@@@@@@@@@@@@@@@@@@@@@ LABEL: Def:(A1)
\begin{itemize}
    \item {\bf Condition (A1)}: 
    \begin{itemize}
        \item[(i)]
        $b$ is differentiable almost everywhere,
        \item[(ii)]
        $\displaystyle b, b^{\prime}\in L_{1} (\mathbf{R})\cap L_{2} (\mathbf{R})\cap L_{\infty} (\mathbf{R})$, 
        \item[(iii)]
        $\Vert b \Vert_{L_{2}(\mathbf{R})} > 0$.
    \end{itemize}
\end{itemize}    
%To discuss a more detailed material in a regularization, we introduce the following condition.
To further elaborate on the analysis in terms of regularization, we also introduce the following condition:
\begin{itemize}
    \item {\bf Condition (A2)}:
    \begin{itemize}
        \item[(1)]
        the function $r$ defined by \eqref{Corr-X}, is fourth differentiable on $\mathbf{R}$ on an open neighborhood of $0$,
        \item[(2)] 
        $\displaystyle r^{(iv)} (0) - ( r^{\prime\prime} (0) )^{2} > 0$, where $r^{(iv)}$ denotes the fourth derivative of $r$.
        \end{itemize} 
        \item[]
\end{itemize}

Under Condition (A1), we have the following well-known properties
(see e.g. \cite{Berman-book} or \cite[p.~123, Section~7.2 and p.~177, Lemma~3--(ii)]{CrLe-book}): 
\begin{Prop} %%%%%%%%%%%%%%%%%%%%
\label{r:diff'ble} %@@@@@@@@@@@@@@@@@@@@ LABEL: r:diff'ble
Assume that $b$ satisfies (A1). Then, $r$ is twice differentiable and
for each $t \in \mathbf{R}$, we have
\begin{itemize}
    \item[{\rm (i)}] $r(t) = r(-t)$,
    \item[{\rm (ii)}] $r^{\prime}(t) = - r^{\prime}(-t)$, in particular, $r^{\prime} (0) = 0$,
    \item[{\rm (iii)}] $0 < - r^{\prime\prime}(0) < +\infty$.
\end{itemize}
\end{Prop} %%%%%%%%%%%%%%%%%%%%%%

%Let $w = \{ w(x) \}_{x \in \mathbf{R}}$ be a one-dimensional Wiener process.
Let us consider a real, continuous and stationary Gaussian process $X = \{ X_{t} \}_{t \in \mathbf{R}}$ defined by \eqref{def-W}. Then, the function $r$ defined by (\ref{Corr-X}) is the covariance function of $X$, $r(t) = \mathrm{Cov} (X_{0}, X_{t})$, and we have the following property. 
\begin{Prop} %%%%%%%%%%%%%%%%%%%%
\label{Stoc-Mdf} %@@@@@@@@@@@@@@@@@@@@@@ LABEL: Stoc-Mdf
Under condition(A1), for almost every $t \in \mathbf{R}$,
$$
\mathbb{P}\big(\text{$X$ is differentiable at $t$}\big)= 1
$$
and
\begin{equation*}
\mathbb{P} \Big(\dot{X}_{t}	= \int_{-\infty}^{\infty} b^{\prime} ( t+x )\,\mathrm{d} w(x)\Big)= 1.
\end{equation*}
\end{Prop} %%%%%%%%%%%%%%%%%%%%%%

Finally, observe that Condition (A2) immediately directly implies the Geman condition, denoted by {\rm (G)}, which we recall here for completeness (see \cite{geman}).
\begin{Def}[{\bf Geman condition}] %%%%%%%%%%
\label{Def:Geman} %@@@@@@@@@@@@@@@@@@@@@ LABEL: Def:Geman
We say that a stationary Gaussian process $\{ X_{t} \}_{t\geqslant 0}$ with a twice differentiable covariance function $r$ satisfies the {\it Geman condition} {\rm (G)} if there exists $\delta > 0$ such that
\begin{equation*}
\frac{r^{\prime\prime}(t) - r^{\prime\prime}(0)}{t}
=: L(t) \in L_{1} ( ( 0, \delta ], \mathrm{d}t ),
\quad \text{i.e.}, \quad
\int_{0+}^{\delta} \vert L(t) \vert \,\mathrm{d}t < +\infty .
\end{equation*}
\end{Def} %%%%%%%%%%%%%%%%%%%%%%%%%%%%%

%%%%%%%%%%%%%%%%%%%%%%%%%%%%%%%%%%%%%%%%
\section{Main Results}
\label{Sec:Results} %@@@@@@@@@@@@@@@@@@@ LABEL: Sec:Results
%%%%%%%%%%%%%%%%%%%%%%%%%%%%%%%%%%%%%%%%

Let $b : \mathbf{R} \to \mathbf{R}$ be square-integrable and let $r$ be the function defined by \eqref{Corr-X}.

From now on, we consider the modification
$\displaystyle \{ W(b^{\prime}_{t}) \}_{t \in \mathbf{R}}$ of $\dot{X}_{t}$, which will still be denoted by
$\displaystyle \dot{X} = \{ \dot{X}_{t} \}_{t \in \mathbf{R}}$.

Let $t \in \mathbf{R}$ be arbitrary.
By
definition, $(X_{t}, \dot{X}_{t})_{t \geqslant 0}$ is a two-dimensional stationary Gaussian process, with independent components (for fixed $t$) since
%(moreover $X_{t}$ and $\dot{X}_{t}$ are independent as
$\displaystyle
\mathbb{E} [ X_{t} \dot{X}_{t} ]
= r^{\prime}(0) = 0
$.
So, the pull-back $\displaystyle \Lambda (X_{t},\dot{X}_{t})$ of $\displaystyle \Lambda \in \mathscr{S}^{\prime} (\mathbf{R}^{2})$ is well defined.
By stationarity, we can write:
$$
\Vert J_{n} [ \Lambda (X_{t}, \dot{X}_{t}) ] \Vert_{L_{2}} =\Vert J_{n} [ \Lambda (X_{0}, \dot{X}_{0})] \Vert_{L_{2}},\quad \forall n  \in \mathbf{Z}^+.
$$
This implies that if 
$\displaystyle \Lambda (X_{0}, \dot{X}_{0}) \in \mathbb{D}_{2}^{\alpha}$,
then, we have
$\displaystyle
\Vert \Lambda (X_{t}, \dot{X}_{t}) \Vert_{2,\alpha} = \Vert \Lambda (X_{0}, \dot{X}_{0})\Vert_{2,\alpha}
$.
Thus,
$\displaystyle
\Lambda (X_{t}, \dot{X}_{t}) \in \mathbb{D}_{2}^{\alpha}
$
and
$\displaystyle
t \in \mathbf{R}  \mapsto \Lambda ( X_{t}, \dot{X}_{t} ) \in \mathbb{D}_{2}^{\alpha}
$
is Bochner integrable on each compact interval.
Hence, we can define the Bochner integral
$$
\int_{0}^{1}
\Lambda ( X_{t}, \dot{X}_{t} )
\,\mathrm{d}t \in \mathbb{D}_{2}^{\alpha}.
$$
In the same way, the Bochner integral
$\displaystyle
\int_{0}^{1} \Lambda ( X_{t} ) \,\mathrm{d}t
$
is defined for $\Lambda \in \mathscr{S}^{\prime} (\mathbf{R})$.

Let us state our main result, of which proof is given in Section~\ref{Sec:Proof}.
\begin{Thm} %%%%%%%%%%%%%%%%%%%%%%%%%%%
\label{Thm:raise} %@@@@@@@@@@@@@@@@@@@@ LABEL: Thm:raise
Let $\alpha \in \mathbf{R}$ be arbitrary.
\begin{itemize}
\item[{\rm (1)}]
Under (A1), for any $\Lambda \in \mathscr{S}^{\prime} (\mathbf{R})$
with $\Lambda ( X_{0} ) \in \mathbb{D}_{2}^{\alpha}$, we have
% there exists a constant $c=c(r) > 0$ such that 
$$
\Vert \int_{0}^{1} \Lambda ( X_{t} ) \,\mathrm{d}t \Vert_{2,\alpha + \frac{1}{2}} \leqslant
c\,\Vert \Lambda ( X_{0} )\Vert_{2,\alpha}
$$
where $c=c(r)$ is a positive constant.
\item[{\rm (2)}]
Under (A2), for any $\Lambda \in \mathscr{S}^{\prime} (\mathbf{R}^{2})$
with $\Lambda ( X_{0}, \dot{X}_{0} ) \in \mathbb{D}_{2}^{\alpha}$, we have
$$
\Vert \int_{0}^{1} \Lambda ( X_{t}, \dot{X}_{t} ) \,\mathrm{d}t \Vert_{2,\alpha + \frac{1}{2}} \leqslant c \,\Vert \Lambda ( X_{0}, \dot{X}_{0} ) \Vert_{2,\alpha}
$$
where $c=c(r)$ is a positive constant.
\end{itemize}
\end{Thm} %%%%%%%%%%%%%%%%%%%%%%%%%%%%%

Theorem~\ref{Thm:raise} can be interpreted as:
`{\it For smooth Gaussian processes, the Bochner integral $\int_{0}^{1}\mathrm{d}t$ regularizes $\Lambda ( X_{t}, \dot{X}_{t} )$ by at least $1/2$},' which can be compared with the result \cite[Corollary~1.6]{amaba:ryu} mentioned in Section~\ref{Intro}.

\vspace{1ex}
We now turn to the application of level-crossing functionals, for which we introduce the notion of {\it Bessel potential spaces}, namely:
$$  
H_{p}^{\alpha}(\mathbf{R}):= ( 1 - \triangle )^{-\alpha /2} L_{p} ( \mathbf{R} ) \quad \text{for}\quad
p \in (1, \infty ) \quad\text{and}\quad \alpha \in \mathbf{R}.
$$
We then obtain the following regularization result, which is derived in Section~\ref{Sec:Proof}:
\begin{Cor} %%%%%%%%%%%%%%%%%%%%%%
\label{Cor:Ocup_Sob} %@@@@@@@@@@@@@@@@@@@@@@@ LABEL: Cor:Ocup_Sob
Under Conditions (A1) and (A2), we have: 
    \begin{itemize}
        \item[(i)] For $f \in H_{p}^{\alpha} (\mathbf{R})$ with $\alpha \in (\frac{1}{2}, 1]$ and $p \in (1, \infty )$,
            \begin{equation*}
            \int_{0}^{1} \delta_{x} ( X_{t} ) f(\dot{X}_{t}) \,\mathrm{d} t \in \mathbb{D}_{2}^{ 0- }.
            \end{equation*}
        \item[(ii)] In particular, the number of crossings of any level $x \in \mathbf{R}$ by the process $X$ on $[0,1]$ satisfies 
        $$
        N(x):=\int_{0}^{1}\delta_{x} ( X_{t} )\vert \dot{X}_{t} \vert\,\mathrm{d}t
        \in \mathbb{D}_{2}^{0-}.
        $$
    \end{itemize}
\end{Cor} %%%%%%%%%%%%%%%%%%%%%%%%

Note that the latter statement (ii)  is weaker than the well-known result in the literature (see e.g. \cite{slud91,slud94,KL-SPA}):
$$
N(x) = \int_{0}^{1} \delta_{x} ( X_{t} ) \vert \dot{X}_{t} \vert \mathrm{d}t \in \mathbb{D}_{2}^{0} = L_{2},
$$
which has been shown to hold under the Geman condition {\rm (G)}  for stationary Gaussian processes with twice differentiable covariance functions, for any level $x \in \mathbf{R}$ (see \cite{KL-AP}).

Thus, (ii) in Corollary~\ref{Cor:Ocup_Sob} does not represent the optimal outcome for the level crossing functional. An alternative approach specifically focused on this functional will be presented in future work.

\section{Examples}
\label{Sec:Examples} %@@@@@@@@@@@@@@@@@@ LABEL: Sec:Examples
%%%%%%%%%%%%%%%%%%%%%%%%%%%%%%%%%%%%%%%%

In this section, we present several covariance functions of Gaussian processes and verify whether Conditions (A1) and (A2) are satisfied. It also illustrates that these conditions are sufficient for Theorem~\ref{Thm:raise}, but not necessary.
Some of these covariance functions belong to important classes
(see e.g., \cite[Chapter~4]{WR06})
that are implemented in \texttt{PyMC} (see \cite{PyMC}). This is why we consider them, to make our theoretical result also relevant in practice.

%%%%%%%%%%%%%%%%%%%%%%%%%%%%%%%%%%%%%%%%
\begin{Eg}{\bf Squared exponential covariance}
\label{Eg:Square} %@@@@@@@@@@@@@@@@@@@@@ LABEL: Eg:Square
\vspace{1ex}

Let $\ell > 0$ and $\displaystyle r(t) := \exp ( - t^{2} / \ell^{2} )$, $t \in \mathbf{R}$.
Then, we have
\begin{equation*}
\begin{split}
F^{\prime} ( \lambda ) = \frac{1}{2 \pi} \int_{-\infty}^{\infty} r(t) \mathrm{e}^{ - i\, \lambda t }\,\mathrm{d}t
=\sqrt{ \frac{ \ell^{2} }{ 4 \pi } } \exp \Big(	- \frac{ \ell^{2} }{ 4 } \lambda^{2}\Big),\quad \lambda \in \mathbf{R},
\end{split}
\end{equation*}
thus,
\begin{equation*}
\begin{split}
b(t) = \int_{-\infty}^{\infty} F^{\prime} ( \lambda )^{1/2} \mathrm{e}^{ i\, \lambda t } \, \mathrm{d} \lambda
=\Big(	\frac{ 4 \pi^{1/2} }{ \ell }\Big)^{1/2} \exp \Big( - \frac{ 2 t^{2} }{ \ell^{2} }\Big) ,
\quad t \in \mathbf{R}.
\end{split}
\end{equation*}
In this case, Conditions (A1) and (A2) are clearly satisfied.
\end{Eg} %%%%%%%%%%%%%%%%%%%%%%%%%%

%%%%%%%%%%%%%%%%%%%%%%%%%%%%%%%%%%%%%%%%
\begin{Eg}{\bf The Mat\'{e}rn class}
\vspace{1ex}

For $\nu > 0$ and $\ell > 0$, let
\begin{equation}
\label{eq:Matern_cov} %@@@@@@@@@@@@@@@@@ LABEL: eq:Matern_cov
r(t):= \frac{ 2^{1-\nu} }{ \Gamma (\nu) } \Big(	\frac{ \sqrt{2\nu} }{ \ell } \vert t \vert \Big)^{\nu} K_{\nu}
\Big( \frac{ \sqrt{2\nu} }{ \ell } \vert t \vert \Big) , \quad t \in \mathbf{R},
\end{equation}
where $K_{\nu}$ is the modified Bessel function of the second kind of order $\nu$.
We assume, without loss of generality, that $r(0) := 1$.
First, let us recall some properties of $K_{\nu}$. Namely, for $\nu > 0$ and $x \in \mathbf{R} \setminus \{ 0 \}$,
\begin{itemize}
\item[(a)]
$ K_{\nu} (x) \underset{x\to 0}{\sim} \frac{ \Gamma ( \nu ) }{ 2 } \left( \frac{x}{2} \right)^{-\nu}$, 
\qquad (b) \qquad 
$K_{\nu} (x) \underset{x\to \infty}{\sim} \sqrt{ \frac{ \pi }{ 2 x } } \, \mathrm{e}^{-x}$,
\item[]
\item[(c)]
$\frac{\partial}{\partial x} ( x^{\nu} K_{\nu}(x) ) = - x^{\nu} K_{\nu -1}(x)$,
\item[]
\item[(d)]
$\frac{\partial K_{\nu}}{\partial x} (x) = \frac{ \nu }{ x } K_{\nu} (x) - K_{\nu +1} (x) =
- K_{\nu -1} (x) - \frac{\nu}{x} K_{\nu} (x)$.
\item[]
\end{itemize}
The first property ensures that $r(t)$ is continuous at $t=0$.
It is known that, for any $m \in \mathbf{Z}^+$,
the function $\vert x \vert^{\nu} K_{\nu} (\vert x \vert)$ is $2m$-times differentiable if and only if $\nu > m$.
Furthermore, a  Gaussian process with this covariance function is ($\lceil \nu \rceil -1$)-times differentiable in the mean-square sense (see \cite{Stein12}).

Let us fix $\nu > 2$, so that $r(t)$ is fourth differentiable. We introduce an auxiliary variable 
$ s = \frac{ \sqrt{ 2 \nu } }{ \ell }\, t$.
% and an auxiliary function defined by
% $\displaystyle\tilde{r}(t) := \vert t \vert^{\nu} K_{\nu} ( \vert t \vert )$.
Then, it holds that
$$
r(t) = \frac{ 2^{ 1 - \nu } }{ \Gamma (\nu ) } s^{\nu} K_{\nu} (s), \quad \text{for}\quad t>0.
$$
Using the above formulas, we can write
\begin{equation*}
\begin{split}
r^{\prime} (t)
&=
\frac{2^{1-\nu}}{\Gamma (\nu)}
\left[ \frac{\mathrm{d}}{\mathrm{d}s} \big( s^{\nu} K_{\nu} (s) \big) \right]
\frac{\mathrm{d}s}{\mathrm{d}t}
= - \,\frac{ \sqrt{2\nu} }{ \ell } \, \frac{ 2^{1-\nu} }{ \Gamma (\nu) } \, s^{\nu} K_{\nu -1}(s), \\
%%%%%%%%%%%%%%%%%%%%%
r^{\prime\prime} (t)
&=
\left( \frac{ \sqrt{2\nu} }{ \ell } \right)^{2} \, \frac{ 2^{1-\nu} }{ \Gamma (\nu) }
\left( - s^{\nu -1} K_{\nu -1}(s) + s^{\nu} K_{\nu -2} (s) \right) 
\; \underset{t \downarrow 0}{\sim} \; 
- \,\frac{\Gamma (\nu -1)}{2 \, \Gamma (\nu )} \left( \frac{ \sqrt{2\nu} }{ \ell } \right)^{2}
= r^{\prime\prime} (0),\\
%%%%%%%%%%%%%%%%%%%%%
r^{\prime\prime\prime} (t)
&=
\left( \frac{ \sqrt{2\nu} }{ \ell } \right)^{3} \, \frac{ 2^{1-\nu} }{ \Gamma (\nu) }
\left( 3 \, s^{\nu -1} K_{\nu -2}(s) - s^{\nu} K_{\nu -3} (s) \right)
\; \underset{t \downarrow 0}{\sim} \; 0 = r^{\prime\prime\prime} (0)
\end{split}
\end{equation*}
To compute the fourth derivative of $r$ at 0, denoted $r^{(iv)}(0)$, we proceed as follows.

As $t \downarrow 0$ (equivalently, as $s \downarrow 0$),
we have
\begin{equation*}
\begin{split}
&
\frac{ r^{\prime\prime\prime} (t) - r^{\prime\prime\prime} (0) }{t} =
\frac{1}{t} r^{\prime\prime\prime} (t) =
\frac{1}{t} \left( \frac{ \sqrt{2\nu} }{ \ell } \right)^{3} \, \frac{ 2^{1-\nu} }{ \Gamma (\nu) }
\left( 3 \, s^{\nu -1} K_{\nu -2}(s) - s^{\nu} K_{\nu -3} (s) \right) \\
& \underset{s \downarrow 0}{\sim} \;
\frac{1}{s}
\left( \frac{ \sqrt{2\nu} }{ \ell } \right)^{4} \, \frac{ 2^{1-\nu} }{ \Gamma (\nu) } 
\Big( 3 \, s^{\nu -1}
\frac{\Gamma (\nu -2)}{2} \left( \frac{s}{2} \right)^{-(\nu -2)} -  s^{\nu}\,
    \frac{\Gamma (\nu -3)}{2} \left( \frac{s}{2} \right)^{-(\nu -3)}
\Big) \\
&=
\left( \frac{ \sqrt{2\nu} }{ \ell } \right)^{4} \,\frac{ 1 }{ \Gamma (\nu) }
\left( 
    \frac{3}{4} \, \Gamma ( \nu -2 ) - \frac{1}{8} \, s^{2} \, \Gamma ( \nu -3 )
\right)
=
\left( \frac{ \sqrt{2\nu} }{ \ell } \right)^{4} \, \frac{ \Gamma ( \nu -2 ) }{ \Gamma (\nu) }\,\frac{3}{4}.
\end{split}
\end{equation*}
Thus, we obtain
$\displaystyle 
r^{(iv)} (0) = \left( \frac{ \sqrt{2\nu} }{ \ell } \right)^{4} \, \frac{ \Gamma ( \nu -2 ) }{ \Gamma (\nu) }
\,\frac{3}{4}$
and
\begin{equation*}
\begin{split}
r^{(iv)} (0) - ( r^{\prime\prime} (0) )^{2}
& =
\left( \frac{ \sqrt{2\nu} }{ \ell } \right)^{4} \,
\frac{ \Gamma ( \nu -2 ) }{ \Gamma (\nu) } \frac{3}{4} -
\left( - \,\frac{ \Gamma (\nu -1) }{2 \, \Gamma (\nu )} \Big( \frac{\sqrt{2\nu}}{\ell} \Big)^{2} \right)^{2} 
\\ &=
\frac{ (2\nu )^{2} }{ 4 \, \ell^{2} } \,
\left( \frac{3}{ \nu ( \nu - 1 ) } - \frac{1}{ ( \nu - 1 )^{2} } \right)
\\ &=
\frac{ (2\nu )^{2} }{ 4 \, \ell^{2} }\,  \frac{ 2 \nu - 3 }{ \nu^{2} ( \nu -1 ) } \,> 0.
\end{split}
\end{equation*}
Therefore, Condition (A2) is satisfied for $\nu > 2$.

\vspace{1ex}
Let us now examine (A1).
By the properties of $K_{\nu}$, it follows that $r \in L_{1} (\mathbf{R})$.
Let us define 
$\gamma := \frac{ \sqrt{ 2 \nu } }{ \ell }$
and
$ \mu := \nu + \frac{1}{2} \,(> \frac{1}{2})$. Then, the distribution function $F(\lambda)$ is determined by the relation
\begin{equation*}
\begin{split}
F^{\prime} ( \lambda )
& = \,\frac{1}{2\pi} \int_{-\infty}^{\infty} r(t) \mathrm{e}^{ - i\, \lambda t }\,\mathrm{d}t
\,= \,\frac{\Gamma ( \mu )}{\pi^{1/2} \Gamma ( \mu - \frac{1}{2} ) \gamma }
\, \Big(1 + \frac{ \lambda^{2} }{ \gamma^{2} } \Big)^{ - \mu },
\quad \lambda \in \mathbf{R}.
\end{split}
\end{equation*}
Therefore, we have
\begin{equation*}
\begin{split}
F^{\prime}( \lambda )^{1/2} =
\, C\, \frac{\Gamma ( \frac{ \mu }{ 2 } )}{\pi^{1/2} \Gamma ( \frac{ \mu }{ 2 } - \frac{1}{2} ) \gamma }
\,\Big(1 + \frac{ \lambda^{2} }{ \gamma^{2} } \Big)^{ - \frac{ \mu }{ 2 } },
\quad
\text{where}\quad 
C :=\frac{\pi^{1/4}\Gamma ( \mu )^{1/2}\Gamma ( \frac{\mu}{2} - \frac{1}{2} )\gamma^{1/2}}
{\Gamma ( \frac{\mu}{2} )\Gamma ( \mu - \frac{1}{2} )^{1/2}}.
\end{split}
\end{equation*}
Thus, if $\frac{\mu}{2} > \frac{1}{2}$, {\it i.e.},
$\nu > \frac{1}{2}$, then, $F^{\prime} \in L_{1}(\mathbf{R})$ and
\begin{equation*}
\begin{split}
b(t) = \int_{-\infty}^{\infty} F^{\prime} ( \lambda )^{1/2} \mathrm{e}^{ i\,\lambda t}\, \mathrm{d} \lambda
\,=\, C\,
\frac{2^{ 1 - ( \frac{\mu}{2} - \frac{1}{2} ) }( \gamma \vert t \vert )^{ \frac{\mu}{2} - \frac{1}{2} }}
{\Gamma ( \frac{\mu}{2} - \frac{1}{2} )}
\,K_{ \frac{\mu}{2} - \frac{1}{2} }( \gamma \vert t \vert ).
\end{split}
\end{equation*}
Now, assume that $\nu > \frac{1}{2}$.
Hence, by expressing $C$, $\gamma$ and $\mu$ in terms of $\nu$ and $\ell$,
we obtain, for $\nu > \frac{1}{2}$,
\begin{equation*}
\begin{split}
b(t)
= \left(\frac{( 2 \nu \pi )^{1/2} \Gamma ( \nu + \frac{1}{2} )}{\Gamma ( \nu )\ell }\right)^{1/2}
\frac{2^{ 1 - \frac{1}{2} ( \nu - \frac{1}{2} ) }}{\Gamma ( \frac{1}{2} ( \nu + \frac{1}{2} ) )}
\left(\frac{\sqrt{ 2 \nu }}{\ell}\vert t \vert \right)^{ \frac{1}{2} ( \nu - \frac{1}{2} ) }
\! K_{ \frac{1}{2} ( \nu - \frac{1}{2} ) }
\left(\frac{ \sqrt{ 2 \nu } }{ \ell } \vert t \vert \right),
\quad t \in \mathbf{R}.
\end{split}
\end{equation*}
By combining properties (a), (b) and (c) above, we conclude that
$b$, $b^{\prime} \in L_{1}(\mathbf{R}) \cap L_{2}(\mathbf{R}) \cap L_{\infty}(\mathbf{R})$,
which implies that Condition (A1) is satisfied for $\nu > \frac{1}{2}$.

\vspace{1ex}
Therefore, we deduce that both Conditions (A1) and (A2) are satisfied for $\nu > 2$.
\end{Eg} %%%%%%%%%%%%%%%%%%%%%%%%%%

%%%%%%%%%%%%%%%%%%%%%%%%%%%%%%%%%%%%%%%%
\begin{Eg}{\bf The Mat\'{e}rn kernel for half integer $\nu$}
\label{Eg:Matern} %@@@@@@@@@@@@@@@@@ LABEL: Eg:Matern
\vspace{1ex}

Suppose that $\nu = m + \frac{1}{2}$ where $m \in \mathbf{Z}^+$.
Then, it is known that the modified Bessel function of the second kind, $K_{\nu}$, is defined by 
\begin{equation*}
K_{\nu} (x) =
\left( \frac{ \pi }{ 2x } \right)^{1/2} \mathrm{e}^{-x} \sum_{k=0}^{m} \frac{ (m+k)! }{ k! (m-k)! } \,(2x)^{-k},
\quad \text{for $x>0$.}
\end{equation*}
According to Equation (\ref{eq:Matern_cov}), we can write
\begin{equation*}
\begin{split}
r(t)
& = \frac{2^{-2m} \pi^{1/2}}{\Gamma (\nu)}
\mathrm{e}^{ - \gamma \vert t \vert }
\sum_{k=0}^{m}\frac{ (m+k)! }{ k! (m-k)! }\,( 2 \gamma \vert t \vert )^{m-k} ,
\quad t \in \mathbf{R},
\end{split}
\end{equation*}
where we recall that $\gamma = \frac{\sqrt{2\nu}}{\ell}$.

\vspace{1ex}
Mat\'{e}rn classes, which are commonly implemented in \texttt{PyMC}, can be described as follows.
\begin{itemize}
\item[(1)]
\texttt{Matern12} (Mat\'{e}rn kernel with $\nu = 1/2$):
$$ 
r(t) = \exp \left( - \frac{\vert t \vert}{\ell} \right).
$$
This example illustrates that  Conditions (A1) or (A2) are sufficient, but not necessary for the regularization effect (see Theorem~\ref{Thm:raise}).
Indeed, we have
\begin{equation*}
\begin{split}
F(\lambda) 
&= \frac{1}{ \pi} \int_{-\infty}^{\lambda} \frac{ \ell }{ ( \ell \, x)^{2} + 1 } \, \mathrm{d}x
= \frac{1}{\pi} \left(\arctan ( \ell \, \lambda ) + \frac{\pi}{2}\right) , 
\quad
%%%%%%%%%%
F^{\prime}(\lambda) = \frac{1}{ \pi}\,\frac{ \ell }{ ( \ell \, \lambda )^{2} + 1 } ,
\end{split}
\end{equation*}
and
\begin{equation*}
\begin{split}
b(x)
&= \frac{1}{2\pi} \int_{-\infty}^{\infty} \sqrt{ \frac{1}{ \pi}\,\frac{ \ell }{ ( \ell \cdot \lambda )^{2} + 1 }}
\, \mathrm{e}^{ i \, x \lambda }\,\mathrm{d}\lambda
=\frac{1}{\pi}\,\frac{ K_{0} ( \vert x \vert / \ell ) }{\sqrt{ \ell \, \pi }}.
\end{split}
\end{equation*}
Hence, $b$ satisfies neither (A1) nor (A2).
However, the associated Gaussian process $X = \{ X_{t} \}_{t \in \mathbf{R}}$ is a stationary Ornstein–Uhlenbeck process with stationary distribution
$\mathbb{P}(X_{t} \in \mathrm{d}x) = ( \pi \ell )^{-1/2} \exp ( - x^{2} / \ell )\mathrm{d}x$,
which is covered in \cite[Corollary~1.6]{amaba:ryu}. 
According to the result, the same regularity effect as in Theorem~\ref{Thm:raise}--(1) is guaranteed.
\item[]
\item[(2)]
\texttt{Matern32} (Mat\'{e}rn kernel with $\nu = 3/2$):
$$
r(t) =\left(1 + \frac{ \sqrt{ 3 t^{2} } }{ \ell } \right) 
\exp \left(-\,  \frac{ \sqrt{ 3 t^{2} } }{ \ell }\right) .
$$
In this example, it is straightforward to verify that $b$ satisfies (A1) but not (A2).
\item[]
\item[(3)]\texttt{Matern52} (Mat\'{e}rn kernel with $\nu = 5/2$):
$$
r(t) = \left(1 + \frac{ \sqrt{ 5 t^{2} } }{ \ell } + \frac{ 5 t^{2} }{ 3 \ell^{2} } \right)
\exp \left(-\, \frac{ \sqrt{ 5 t^{2} } }{ \ell }\right) .
$$
In this example, $b$ satisfies both (A1) and (A2).
\end{itemize}
\end{Eg} %%%%%%%%%%%%%%%%%%%%%%%%%%

%%%%%%%%%%%%%%%%%%%%%%%%%%%%%%%%%%%%%%%%
\begin{Eg}{\bf $\gamma$-exponential covariance}
\vspace{1ex}

For $\ell > 0$ and $\gamma \in (0,2]$, let
$$
r(t):=\exp(- ( \vert t \vert / \ell )^{\gamma}), \quad t \in \mathbf{R}.
$$
Let $r^{\prime}(t)$ and $r^{\prime\prime}(t)$ for $t \geqslant 0$ denote the first and second right hand derivatives, respectively.

Then, $r^{\prime\prime} (0) \neq 0$ if and only if $\gamma = 1$ or $2$.
The case $\gamma = 2$ corresponds to Example~\ref{Eg:Square},
while the case $\gamma = 1$ is precisely Example~\ref{Eg:Matern}--(1).

\end{Eg} %%%%%%%%%%%%%%%%%%%%%%%%%%

%%%%%%%%%%%%%%%%%%%%%%%%%%%%%%%%%%%%%%%%
\begin{Eg}{\bf Rational quadratic covariance}
\vspace{1ex}

For $\ell , \alpha > 0$, define
$$
r(t) :=\left(1 + \frac{ t^{2} }{ 2 \alpha \ell^{2} }\right)^{-\alpha},\quad t \in \mathbf{R}.
$$
Then, we can write
\begin{equation*}
\begin{split}
r^{\prime} (t)
&=
- \left( 1 + \frac{t^{2}}{2\alpha \ell^{2}} \right)^{-\alpha -1} \frac{t}{\ell^{2}} , \\
%%%%%%%%%%%
r^{\prime\prime} (t)
&=
\left( 1 + \frac{1}{\alpha} \right)
\left( 1 + \frac{t^{2}}{2\alpha \ell^{2}} \right)^{-\alpha -2} \frac{ t^{2} }{ \ell^{4} }
\, - \,\left( 1 + \frac{t^{2}}{2\alpha \ell^{2}} \right)^{-\alpha -1} \frac{1}{\ell^{2}} , \\
%%%%%%%%%%%
r^{\prime\prime\prime} (t)
&=
- \left( 1 + \frac{1}{\alpha} \right) \left( 1 + \frac{2}{\alpha} \right)
\left( 1 + \frac{t^{2}}{2\alpha \ell^{2}} \right)^{ - \alpha - 3 } \!\left( \frac{ t }{ \ell^{2} } \right)^{3} 
+\, 3 \left( 1 + \frac{1}{\alpha} \right)
\left( 1 + \frac{t^{2}}{2\alpha \ell^{2}} \right)^{ - \alpha - 2 }\!\!\frac{t}{ ( \ell^{2} )^{2} } , \\
%%%%%%%%%%%
r^{(iv)} (t)
&=
\left( 1 + \frac{1}{\alpha} \right)
\left( 1 + \frac{2}{\alpha} \right)
\left( 1 + \frac{3}{\alpha} \right)
\left( 1 + \frac{t^{2}}{2\alpha \ell^{2}} \right)^{ - \alpha - 4 }
\left( \frac{ t }{ \ell^{2} } \right)^{4} \\
&\hspace{8mm}
- 6 \left( 1 + \frac{1}{\alpha} \right)\left( 1 + \frac{2}{\alpha} \right)
\left( 1 + \frac{t^{2}}{2\alpha \ell^{2}} \right)^{ - \alpha - 3 }
\!\frac{ t^{2} }{ ( \ell^{2} )^{3} } 
\,+\, 3 \left( 1 + \frac{1}{\alpha} \right) 
\left( 1 + \frac{t^{2}}{2\alpha \ell^{2}} \right)^{ - \alpha - 2 }
\!\!\frac{ 1 }{ ( \ell^{2} )^{2} }.
\end{split}
\end{equation*}
Therefore, 
$\displaystyle r^{\prime\prime} (0) = - \frac{ 1 }{ \ell^{2} } $
and
$\displaystyle r^{(iv)} (0)= \frac{3}{\ell^{4}}\left( 1 + \frac{1}{\alpha} \right)$,
and
\begin{equation*}
\begin{split}
r^{(iv)} (t) - ( r^{\prime\prime} (0) )^{2}
&=
\frac{3}{\ell^{4}}\left( 1 + \frac{1}{\alpha} \right)
- \left( - \frac{ 1 }{ \ell^{2} } \right)^{2}
= \frac{ 2 \alpha + 3 }{ \ell^{4} }
> 0.
\end{split}
\end{equation*}
Hence, (A2) holds.

Now, let us  examine whether Condition (A1) is satisfied.
Suppose that $\alpha > 1$. Then, $\displaystyle r \in L_{1}(\mathbf{R})$, so that
the associated spectral distribution $F(\lambda)$ is given by:
\begin{equation*}
F^{\prime} ( \lambda )
= \frac{1}{2\pi} \int_{-\infty}^{\infty} r(t) \mathrm{e}^{ - i \, \lambda t } \,\mathrm{d}t,
\quad \lambda \in \mathbf{R}.
\end{equation*}
For the integrability, observe that
$$
\left\vert \left( \frac{\mathrm{d}^{n}}{\mathrm{d}t^{n}} r(t) \right) \right\vert
= \left\vert 
\frac{\mathrm{d}^{n}}{\mathrm{d}t^{n}} \left( 1 + \frac{ t^{2} }{ 2 \alpha \ell^{2} }\right)^{-\alpha}
\right\vert 
\leqslant
C_{n} \left( 1 + \frac{ t^{2} }{ 2 \alpha \ell^{2} } \right)^{-\alpha} =\, C_{n}r(t)
$$
for some positive constant $C_{n}$.
Additionally, we have
\begin{equation*}
(-i \, \lambda)^{n} F^{\prime} ( \lambda )
\,= \,
\frac{1}{2\pi} \int_{-\infty}^{\infty} r(t)
\left( \frac{\mathrm{d}^{n}}{\mathrm{d}t^{n}} \mathrm{e}^{ - i \, \lambda t } \right) \,\mathrm{d}t \\
\, =\,  %%%%%%%%%%%%%
\frac{ (-1)^{n} }{ 2\pi }\int_{-\infty}^{\infty}\left( \frac{\mathrm{d}^{n}}{\mathrm{d}t^{n}} r(t) \right)
\mathrm{e}^{ - i \, \lambda t }\,\mathrm{d}t.
\end{equation*}
Thus, we obtain
$$
\vert \lambda\vert^{n} F^{\prime} (\lambda) <
C_{n}\, \frac{1}{2\pi} \int_{-\infty}^{\infty} \vert r(t) \mathrm{e}^{-i \, \lambda t} \vert\,\mathrm{d}t
\,=\, C_{n} \,F^{\prime}(0) < \infty
$$
for any $n \in \mathbf{N}$.
Therefore,
$$
F^{\prime} ( \lambda )^{1/2} \in \cap_{p\geqslant 1}L_{p}(\mathbf{R}, \mathrm{d}\lambda),
$$
and it follows that the function $b$ defined by Equation (\ref{eq:b}) is differentiable and satisfies (A1).

\end{Eg} %%%%%%%%%%%%%%%%%%%%%%%%%%

%%%%%%%%%%%%%%%%%%%%%%%%%%%%%%%%%%%%%%%%
\begin{Eg}{\bf Wendland's piecewise polynomial with compact support} 
\vspace{1ex}

Let $\mathcal{C}$ denote the set of all continuous functions $\varphi : [0, +\infty ) \to \mathbf{R}$
with compact support in $[0,1]$.
Define $ \mathcal{I} : \mathcal{C} \to \mathcal{C} $
by
$$
( \mathcal{I} \varphi ) (t) := \int_{t}^{+\infty} s \, \varphi (s) \,\mathrm{d}s, \quad t \geqslant 0.
$$
For each $n,s,k \in \mathbf{N}$, define $\varphi_{n}, \varphi_{s,k} \in \mathcal{C}$
by
$$
\varphi_{n} (t)
:=
\mathbf{1}_{\{ 0 \leqslant t \leqslant 1\}}
( 1 - t )^{n}, \quad t \geqslant 0
$$
and
$$
\varphi_{s,k} := \mathcal{I}^{k} \varphi_{ \lfloor s/2 \rfloor + k + 1 }.
$$
By the definition of $\mathcal{I}$ and since $\varphi_{n}(t)$ is non-increasing in $t$,
we conclude that $\varphi_{s,k} (0) > 0$ and that $\varphi_{s,k}$ is also non-increasing.

Moreover, it is known that the function $\Phi_{s,k}:\mathbf{R}^{s} \to \mathbf{R}$ defined by
$$
\Phi_{s,k} ( x )= \varphi_{s,k} ( \vert x \vert )
$$
is $2k$-times continuously differentiable and strictly positive definite (see \cite{We05}).

Now, letting $s=1$ and $k \in \mathbf{N}$,
we define
$$
r(t) := \Phi_{1,k} (t) / \Phi_{1,k} (0) = c \,( \mathcal{I}^{k} \varphi_{k+1} ) (t), \quad t \geqslant 0,
$$
where $c = ( \varphi_{1,k} (0) )^{-1} > 0$.
Since $r(\cdot)$ has a compact support, Condition (A1) is satisfied.
\vspace{1ex}
Next, we turn to (A2).
For $k \geqslant 4$, it holds that
\begin{equation*}
\begin{split}
c^{-1} \, r^{\prime}(t)
&=
- t \, ( \mathcal{I}^{k-1} \varphi_{k+1} ) (t), \\
%%%%%%%%%%%
c^{-1} \, r^{\prime\prime}(t)
&=
- ( \mathcal{I}^{k-1} \varphi_{k+1} ) (t) + t^{2} \, ( \mathcal{I}^{k-2} \varphi_{k+1} ) (t), \\
%%%%%%%%%%%
c^{-1} \, r^{\prime\prime\prime}(t)
&=
t \, ( \mathcal{I}^{k-2} \varphi_{k+1} ) (t)
+ 2 t \, ( \mathcal{I}^{k-2} \varphi_{k+1} ) (t)
- t^{3} \, ( \mathcal{I}^{k-3} \varphi_{k+1} ) (t), \\
%%%%%%%%%%%
c^{-1} \, r^{(iv)}(t)
&=
3 ( \mathcal{I}^{k-2} \varphi_{k+1} ) (t)
- 6 t^{2} \, ( \mathcal{I}^{k-3} \varphi_{k+1} ) (t)
+ t^{4} \, ( \mathcal{I}^{k-4} \varphi_{k+1} ) (t),
\end{split}
\end{equation*}
from which we deduce
\begin{equation*}
r^{\prime\prime}(0)
=  - \,\frac{ ( \mathcal{I}^{k-1} \varphi_{k+1} ) (0)}{( \mathcal{I}^{k} \varphi_{k+1} ) (0)}
\quad\text{and}\quad 
r^{(iv)}(0)
= 3 \frac{( \mathcal{I}^{k-2} \varphi_{k+1} ) (0)}{( \mathcal{I}^{k} \varphi_{k+1} ) (0)}.
\end{equation*}
Here, we compute $( \mathcal{I}^{n} \varphi_{k+1} )(0)$. Using integration by parts, we have
\begin{equation*}
\begin{split}
( \mathcal{I}^{n} \varphi_{k+1} )(0)
& = \int_{0}^{1} s \, ( \mathcal{I}^{n-1} \varphi_{k+1} )(s) \,\mathrm{d}s
= \int_{0}^{1} \frac{ s^{3} }{2} \,( \mathcal{I}^{n-2} \varphi_{k+1} )(s)\,\mathrm{d}s \\
& = \cdots
= \int_{0}^{1}\frac{ s^{2n-1} }{ (2(n-1))!}\,\varphi_{k+1}(s)\,\mathrm{d}s
= \frac{ \mathrm{B}(2n,k+2) }{ 2^{n-1} \Gamma (n)},
\end{split}
\end{equation*}
where $\mathrm{B}$ denotes the Beta function and $\Gamma$ the Gamma function.
From this, we obtain
\begin{equation*}
\begin{split}
&
r^{(iv)} (0) - ( r^{\prime\prime} (0) )^{2} \\
& = 3 \,\frac{ \mathrm{B}(2(k-2),k+2) }{ 2^{k-3} \Gamma (k-2)}\, \frac{ 2^{k-1} \Gamma (k) }{ \mathrm{B}(2k,k+2) }
- \left(-\, \frac{ \mathrm{B}(2(k-1),k+2) }{ 2^{k-2} \Gamma (k-1)}\frac{ 2^{k-1} \Gamma (k)}{ \mathrm{B}(2k,k+2) }\right)^{2} \\
& = \frac{ 18 k (3k+1) \big( 2k ( k ( 3k - 5 ) + 4 ) - 1 \big)}{ ( 2k-1 )^{2} ( 2k-3 )} > 0.
\end{split}
\end{equation*}
Hence (A2) is satisfied.

\end{Eg} %%%%%%%%%%%%%%%%%%%%%%%%%%

%%%%%%%%%%%%%%%%%%%%%%%%%%%%%%%%%%%%%%%%
\begin{Eg}{\bf The Cosine kernel}
\vspace{1ex}

The covariance function given by
\begin{equation*}
r(t) = \cos \Big(\pi \frac{ \vert t \vert }{ \ell^{2} } \Big)
\end{equation*}
is not square-integrable.
Even if a covariance function of the above form is considered around $t=0$ such that Condition (A1) is satisfied, Condition (A2) is not valid, since
$\displaystyle r^{(iv)} (0) - ( r^{\prime\prime} (0) )^{2} = 0$.
\end{Eg} %%%%%%%%%%%%%%%%%%%%%%%%%%

%%%%%%%%%%%%%%%%%%%%%%%%%%%%%%%%%%%%%%%%
\begin{Eg}{\bf The Period kernel}
\vspace{1ex}

Let $T>0$. The covariance function defined by
\begin{equation*}
r(t) = \exp \Big(-\, \frac{ \sin^{2} ( \pi \vert t \vert / T ) }{ \ell^{2} } \Big)
\end{equation*}
is not square-integrable.
If a covariance function of this form is considered around $t=0$,
Condition (A2) is satisfied, since a straightforward computation shows that
\begin{equation*}
r^{(iv)} (0) - ( r^{\prime\prime} (0) )^{2} = \frac{ 8 \pi^{4} }{ ( T \ell )^{4} } ( \ell^{2} + 1 ) > 0.
\end{equation*}
\end{Eg} %%%%%%%%%%%%%%%%%%%%%%%%%%

%%%%%%%%%%%%%%%%%%%%%%%%%%%%%%%%%%%%%%%%
\section{Proofs}
\label{Sec:Proof} %@@@@@@@@@@@@@@@@@@@@@ LABEL: Sec:Proof

%%%%%%%%%%%%%%%%%%%%%%%%%%%%%%%%%%%%%%%%
% \subsection{Proof of Proposition~\ref{r:diff'ble}}
% %%%%%%%%%%%%%%%%%%%%%%%%%%%%%%%%%%%%%%%%
%
% %%%%%%%%%%%%%%%%%%%%%%%%%%%%%%%%%%%%%%%%
% \begin{proof}[Proof of Proposition~\ref{r:diff'ble}]
% First, we shall prove that $r$ is differentiable.
% Let $t \in \mathbf{R}$ be arbitrary.
% Then, after applying a change of variable,
% we have
% \begin{equation*}
% \frac{
% 	r(t+h) - r(t)
% }{ h }
% =
% \int_{-\infty}^{\infty}
% \Big(
% \frac{1}{h}
% \int_{0}^{h}
% b^{\prime}( x+s )
% \mathrm{d}s
% \Big)
% b(x-t)
% \mathrm{d}x.
% \end{equation*}
% Since
% $
% \vert
% h^{-1}
% \int_{0}^{h}
% b^{\prime}( x+s )
% \mathrm{d}s
% \vert
% \leqslant
% \Vert b^{\prime} \Vert_{\infty}
% < +\infty
% $
% and $b \in L_{1} (\mathbf{R})$,
% we can apply the dominated convergence theorem
% to conclude
% $
% \lim_{h\to 0}\frac{r(t+h) - r(t)}{h}
% =
% \int_{-\infty}^{\infty}
% b^{\prime} (x)
% b(x-t)
% \mathrm{d}x
% $.
% The differentiablity of $r^{\prime}$ can be proved similarly.

% Now (i) and (ii) are clear.
% We give a proof of (iii).
% Since
% $b, b^{\prime} \in L_{2}(\mathbf{R})$,
% we have
% $
% r^{\prime\prime} (0)
% =
% -
% \int_{-\infty}^{\infty}
% b^{\prime} (x)
% b^{\prime}(x)
% \mathrm{d}x
% =
% - \Vert b^{\prime} \Vert_{L_{2}(\mathbf{R})}^{2}
% \leqslant 0
% $.
% Hence if
% $
% r^{\prime\prime} (0) = 0
% $,
% the function $b$
% must be a constant which is nonzero because of
% (A1)--(3),
% and this contradicts to $b \in L_{2}(\mathbf{R})$.
% Therefore $- r^{\prime\prime} (0) > 0$.
% \end{proof} %%%%%%%%%%%%%%%%%%%%%%%%%%%%

%%%%%%%%%%%%%%%%%%%%%%%%%%%%%%%%%%%%%%%%
\subsection{Proof of Proposition~\ref{Stoc-Mdf}}
%%%%%%%%%%%%%%%%%%%%%%%%%%%%%%%%%%%%%%%%

%%%%%%%%%%%%%%%%%%%%%%%%%%%%%%%%%%%%%%%%
%\begin{proof}[Proof of Proposition~\ref{Stoc-Mdf}]
Let 
$$ 
E := \{ \text{$X$ is absolutely continuous}\}.
$$
By \cite[Proposition~4]{It63/64}, we have 
$\displaystyle \mathbb{P} (E) = 1$ (on the completed probability space).
For each $t \in \mathbf{R}$,
define
$$
F(t) := \{\text{$X$ is not differentiable at $t$} \}.
$$
Then $\displaystyle \mathbf{1}_{F(t)}( \omega )$ is measurable with respect to $(t, \omega)$. For any $\omega \in E$, we have $\displaystyle \mathbf{1}_{F(t)}( \omega ) = 0$ for almost all $t \in \mathbf{R}$, so
$\displaystyle \int_{\mathbf{R}} \mathbf{1}_{F(t)}( \omega ) \,\mathrm{d}t = 0 $.
Then, we have
$$
0 
= \mathbb{E}[\int_{\mathbf{R}} \mathbf{1}_{F(t)} \,\mathrm{d}t ] 
= \int_{\mathbf{R}} \mathbb{E}[ \mathbf{1}_{F(t)}] \,\mathrm{d}t 
= \int_{\mathbf{R}} \mathbb{P} (\text{$X$ is not differentiable at $t$}) \,\mathrm{d}t ,
$$
from which the first assertion follows.

For the second assertion, let $t \in \mathbf{R}$ be such that $X$ is differentiable at $t$ almost surely.
Then, applying Fatou's lemma gives
\begin{equation*}
\begin{split}
& \mathbb{E}
\left[\left(\dot{X}_{t} - \int_{-\infty}^{\infty} b^{\prime} ( t+x ) \,\mathrm{d} w(x) \right)^{2}\right]\\
& \leqslant 
\liminf_{h \to 0} \mathbb{E}
\left[ \left( \int_{-\infty}^{\infty} \Big\{ \frac{ b ( t+h+x ) - b (t+x) }{ h } - b^{\prime} (t+x) \Big\} \,\mathrm{d} w(x) \right)^{2}\right].
\end{split}
\end{equation*}
Then, the It\^{o} isometry leads to
\begin{equation*}
\mathbb{E}
\left[ \left(\dot{X}_{t} - \int_{-\infty}^{\infty} b^{\prime} ( t+x ) \,\mathrm{d} w(x) \right)^{2}\right]
\leqslant
\liminf_{h \to 0} \int_{-\infty}^{\infty} \Big\{\frac{ b ( x+h ) - b (x) }{ h } - b^{\prime} (x) \Big\}^{2} \,\mathrm{d} x.
\end{equation*}
By Jensen's inequality, we have
\begin{equation*}
F_{h} (x) :=
\left\{\frac{ b ( x+h ) - b (x) }{ h } - b^{\prime} (x) \right\}^{2}
\leqslant
\frac{1}{h} \int_{0}^{h} \{ b^{\prime} (x+s) - 
 b^{\prime} (x) \}^{2} \mathrm{d}s =: G_h (x).
\end{equation*}
By the fundamental theorem of calculus for Lebesgue integrals, we find that
$$
\frac1h \int_{0}^{h} b^{\prime} (x+s) \,\mathrm{d}s \underset{h\to 0}{\to} b^{\prime} (x)
\quad
\text{for almost every $x \in \mathbf{R}$.}
$$
Furthermore, we have
\begin{equation*}
\int_{-\infty}^{\infty} G_{h} (x) \,\mathrm{d}x
= 2 \int_{-\infty}^{\infty} b^{\prime} (x)^{2} \,\mathrm{d}x - 2 \int_{-\infty}^{\infty} \Big( \frac{1}{h} \int_{0}^{h} b^{\prime} (x+s) \,\mathrm{d}s \Big) \,b^{\prime} (x) \,\mathrm{d}x .
\end{equation*}
But, under Conditions (A1) and (A2), it holds that
$\displaystyle \left\vert \frac1h \int_{0}^{h} b^{\prime} (x+s) \,\mathrm{d}s \right\vert
\leqslant \Vert b^{\prime} \Vert_{\infty} $ 
and $\displaystyle b^{\prime} \in L_{1}(\mathbf{R})$.
Applying the dominated convergence theorem gives
$$
\int_{-\infty}^{\infty} \left(
\frac1h \int_{0}^{h} b^{\prime} (x+s) \,\mathrm{d}s\right) \,b^{\prime} (x) \,\mathrm{d}x
\underset{h\to 0}{\to}
\Vert b^{\prime} \Vert_{L^{2}(\mathbf{R})}^{2}.
$$
Thus, we obtain
\begin{equation*}
\mathbb{E}
\left[ \Big( \dot{X}_{t} - \int_{-\infty}^{\infty} b^{\prime} ( t+x ) \,\mathrm{d} w(x) \Big)^{2}\right]
\leqslant
\liminf_{h \to 0} \int_{-\infty}^{\infty} F_{h} (x) \,\mathrm{d} x 
\leqslant
\liminf_{h \to 0} \int_{-\infty}^{\infty} G_{h} (x) \,\mathrm{d} x = 0. \qquad \Box
\end{equation*}
%\end{proof} %%%%%%%%%%%%%%%%%%%%%%%%%%%%

%%%%%%%%%%%%%%%%%%%%%%%%%%%%%%%%%%%%%%%%
\subsection{Proof of Theorem~\ref{Thm:raise}}
%%%%%%%%%%%%%%%%%%%%%%%%%%%%%%%%%%%%%%%%

Suppose that a Gaussian process $\{ X_{t} \}_{t \in \mathbf{R}}$ is given by Equation \eqref{def-W},
where $b : \mathbf{R} \to \mathbf{R}$ is square-integrable.
Then the associated covariance function $r$ satisfies \eqref{Corr-X}.
% is given by
% \begin{equation*}
% r(t)
% :=
% \int_{-\infty}^{\infty}
% b(t+x)
% b(x)
% \,\mathrm{d}x,
% \quad
% t \in \mathbf{R}.
% \end{equation*}
We assume, without loss of generality, that
$r(0) = \Vert b \Vert_{L_{2}(\mathbf{R})}^{2} = 1$
(since $X_t$ can be replaced by $X_t/\sqrt{\mathrm{Var}(X_t)}$).

\vspace{1ex}
To prove Theorem~\ref{Thm:raise}, we first establish two intermediate technical results, which may be of independent interest: Lemma~\ref{Lem:iter-asym}, which will be applied in the first part of the proof of Theorem~\ref{Thm:raise}, and Lemma~\ref{Prop:quad_est}, which will contribute to the second part.

\subsubsection{Technical lemmas}
\label{sss:tech-Lem}

\begin{Lem} %%%%%%%%%%%%%%%%%%%%%%%%%%
\label{Lem:iter-asym} %@@@@@@@@@@@@@@@@@ LABEL: Lem:iter-asym
For any $c > 0$ and $c^{\prime} \in ( 0, c^{-1/2} )$,
we have
\begin{equation*}
\iint_{ 0 \leqslant s < t \leqslant c^{\prime} }
( 1 - c \, s^{2} )^{n} \,\mathrm{d}s \,\mathrm{d}t
\,=\,  O(n^{-1/2}) \quad \text{as $n \to \infty$.}
\end{equation*}
\end{Lem} %%%%%%%%%%%%%%%%%%%%%%%%%%%%

\begin{proof} %%%%%%%%%%%%%%%%%%%%%%%%%%
We have
$$
\int_{0}^{c^{\prime}} \int_{0}^{t} ( 1 - c \, s^{2} )^{n} \,\mathrm{d}s \mathrm{d}t\,
= c^{-1/2} \int_{0}^{c^{\prime}} \int_{0}^{t\sqrt{c}} ( 1 - s^{2} )^{n} \,\mathrm{d}s\, \mathrm{d}t .
$$
Computing this last integral with series (see e.g. \cite{GradR-book}) gives
\begin{equation*}
\begin{split}
\int_{ 0 }^{ t \sqrt{c} } ( 1 - s^{2} )^{n} \,\mathrm{d}s
&= %%%%%%%
\sum_{k=0}^{n} \binom{n}{k} \frac{ (-1)^{k} }{ 2k+1 }
( c^{1/2} t )^{2k+1} \\
&= %%%%%%%
c^{1/2} t \sum_{k=0}^{\infty} \frac{( \frac{1}{2} )_{k}( -n )_{k}}{( \frac{3}{2} )_{k}}\frac{ ( c t^{2} )^{k} }{k!} =( c^{1/2} t )\cdot
\mbox{}_{2}F_{1}
({\textstyle \frac{1}{2}}, -n ;{\textstyle \frac{3}{2}}; c t^{2}),
\end{split}
\end{equation*}
where $(q)_{k}$ denotes the (rising) Pochhammer symbol
and $\mbox{}_{2}F_{1} (a,b;c;z)$ the hypergeometric function.
Thus, we obtain
\begin{equation*}
\int_{0}^{c^{\prime}} \int_{0}^{t} ( 1 - c \, s^{2} )^{n} \,\mathrm{d}s\, \mathrm{d}t
%%%%%%%%%%
\leqslant
\int_{0}^{c^{\prime}} t \cdot \mbox{}_{2}F_{1} ({\textstyle \frac{1}{2}}, -n ; {\textstyle \frac{3}{2}}; c t^{2} ) \,\mathrm{d}t
%%%%%%%%%%
\leqslant
c^{-1/2}\int_{0}^{1} t \cdot \mbox{}_{2}F_{1}
({\textstyle \frac{1}{2}}, -n ; {\textstyle \frac{3}{2}}; t^{2}) \,\mathrm{d}t .
\end{equation*}
We use the fact that
$\displaystyle
\frac{\mathrm{d}}{\mathrm{d} t}
\hspace{1mm} \mbox{}_{2}F_{1} ( a, b; c; t )
= (ab/c) \hspace{1mm} \mbox{}_{2}F_{1} ( a+1, b+1; c+1; t )
$, we can write
\begin{equation*}
t \cdot \mbox{}_{2}F_{1}
({\textstyle \frac{1}{2}}, -n ; {\textstyle \frac{3}{2}}; t^{2} )
=
\frac{1}{2(n+1)} \frac{ \mathrm{d} }{ \mathrm{d} t } \hspace{1mm} \mbox{}_{2}F_{1} ( {\textstyle -\frac{1}{2}}, -n-1 ; {\textstyle \frac{1}{2}}; t^{2} ),
\end{equation*}
and so, 
\begin{equation*}%\label{eq:int-hypergeom}
\int_{0}^{1-} t \cdot \mbox{}_{2}F_{1}
( {\textstyle \frac{1}{2}}, -n ; {\textstyle \frac{3}{2}}; t^{2} ) \,\mathrm{d} t
=
\frac{1}{2(n+1)} \Big\{ \mbox{}_{2}F_{1}
({\textstyle - \frac{1}{2}}, -n-1 ; {\textstyle \frac{1}{2}}; 1 ) - 1 \Big\}.
\end{equation*}
Noticing that $b=-n-1$ is a nonpositive integer, then $\mbox{}_{2}F_{1} (a,b;c;z)$ is a polynomial in $z$
(recall that the hypergeometric function $\mbox{}_{2}F_{1} (a,b;c;z)$ is generally defined as a convergent series in $\vert z \vert < 1$;
see e.g. \cite{GradR-book}),
hence the last identity is still valid even in the case $z=1$.
Then, since
$ a+b = - \frac{1}{2} + (-n-1 ) < \frac{1}{2} = c $,
the Gauss hypergeometric theorem tells us that
\begin{equation*}
\mbox{}_{2}F_{1}
({\textstyle - \frac{1}{2}}, -n-1 ; {\textstyle \frac{1}{2}}; 1 )
=
\frac{\Gamma ( \frac{1}{2} )\Gamma ( n + 2 )}{\Gamma ( 1 )\Gamma ( n + \frac{3}{2} )}
=
\sqrt{\pi}\, \frac{\Gamma ( n + 2 )}{\Gamma ( n + \frac{3}{2} )},
\end{equation*}
which behaves, by using Stirling's formula, as
\begin{equation*}
\sqrt{\pi} \frac{ \Gamma ( n + 2 ) }{\Gamma ( n + \frac{3}{2} )}
\approx
\sqrt{\pi} \frac{\sqrt{\frac{2\pi}{n+2}}\Big(\frac{n+2}{\mathrm{e}}\Big)^{n+2}}{
 \sqrt{\frac{2\pi}{n+\frac{3}{2}}} \Big(\frac{n+\frac{3}{2}}{\mathrm{e}}\Big)^{n+\frac{3}{2}}}
=
\sqrt{ \frac{\pi}{\mathrm{e}} }\frac{(n+2)^{ n + \frac{3}{2} }}{(n+\frac{3}{2})^{ n + 1 }}
= O(n^{1/2})
\quad
\text{as $n\to \infty$.}
\end{equation*}
We can then conclude that
\begin{equation*}
\int_{ 0 \leqslant s < t \leqslant c^{\prime}
} ( 1 - c \, s^{2} )^{n} \,\mathrm{d}s
\mathrm{d}t = O (n^{-1/2})\quad
\text{as $n\to\infty$.}
\end{equation*}
\end{proof} %%%%%%%%%%%%%%%%%%%%%%%%%%%%

Let us move on to our second lemma that is needed to prove the second part of Theorem~\ref{Thm:raise}, for which we assume that the correlation function $r$ satisfies Condition (A2).

For a matrix $M=(M_{ij})_{ij}$, we denote by
$$
\Vert M \Vert_{\mathrm{HS}} :=
\mathrm{tr} (M\hspace{0.5mm}\mbox{}^{t}\!M)
=\sqrt{ \sum_{ij} M_{ij}^{\,2} }
$$
the Hilbert--Schmidt norm of $M$. 

Introducing
\begin{equation}\label{def:matrixA}
\begin{split}
A (t)
&=
\left(\begin{array}{cc}
A_{11} (t) & A_{12} (t) \\
A_{21} (t) & A_{22} (t)
\end{array}\right)
:=
\left(\begin{array}{cc}
r (t) & {\displaystyle - \,r^{\prime} (t)/ \sqrt{ - r^{\prime\prime} (0)}}\\
{\displaystyle r^{\prime} (t) / \sqrt{ - r^{\prime\prime} (0)}}
&
{\displaystyle -\, r^{\prime\prime} (t)/(- r^{\prime\prime} (0)) } 
\end{array}\right),
\end{split}
\end{equation}
we can prove the following:
\begin{Lem}
\label{Prop:quad_est} %@@@@@@@@@@@@@@@@@ LABEL: Prop:quad_est
Under Condition (A2),
there exist $c, c^{\prime}>0$ such that $A$ defined in \eqref{def:matrixA} satisfies:
\begin{equation*}
\Vert A(t) \Vert_{\mathrm{HS}}
\leqslant
1 - c \, t^{2}
\quad
\text{for $0 \leqslant t \leqslant c^{\prime}$.}
\end{equation*}
\end{Lem} %%%%%%%%%%%%%%%%%%%%%%

\begin{proof}
Since $r^{\prime}$ and $r^{\prime\prime\prime}$ are odd functions,
it is straightforward to prove that $\Vert A(t) \Vert_{\mathrm{HS}}^{2}$ at $t=0$ has the following variations:
%%%%%%%%%%%%%%%%%%%%%%%%%%%%%%%%%%%%%%%%
%\begin{Lem}[Variations of $\Vert A(t) \Vert_{\mathrm{HS}}^{2}$ at $t=0$]
%We have
\begin{itemize}
\item[{\rm (1)}]
$\displaystyle
\left.\frac{ \mathrm{d} }{ \mathrm{d} t }\right\vert_{t=0}
\left(\Vert A(t) \Vert_{\mathrm{HS}}^{2}\right)
= \left.\frac{ \mathrm{d} }{ \mathrm{d} t }\right\vert_{t=0}
\left( r(t)^{2} + 2 \left(
\frac{ r^{\prime} (t) }{ \sqrt{- r^{\prime\prime} (0)}}\right)^{2} 
+ \left( \frac{ - r^{\prime\prime} (t) }{ - r^{\prime\prime} (0) } \right)^{2} \right) = 0
$.
\item[{\rm (2)}]
$\displaystyle
\left.\frac{ \mathrm{d}^{2} }{ \mathrm{d} t^{2} }\right\vert_{t=0}
\left(\Vert A(t) \Vert_{\mathrm{HS}}^{2}\right)
=
\frac{ 1 }{ r^{\prime\prime} (0) }
\big(r^{(iv)} (0) - ( r^{\prime\prime} (0) )^{2}\big)
$.
\end{itemize}
%\end{Lem} %%%%%%%%%%%%%%%%%%%%%%%%%%%%
%The proof is straightforward and hence omitted.
%
Now, under Condition (A2), using Taylor's expansion, we can conclude that there exists $c^{\prime} > 0$ such that
\begin{equation*}
\Vert A(t) \Vert_{\mathrm{HS}}^{2}
= 1 - 2 c^{\prime} \, t^{2} + O(t^{4})
\quad \text{as $t\to 0$,}
\end{equation*}
and hence,
\begin{equation*}
\Vert A(t) \Vert_{\mathrm{HS}}
= 1 - c^{\prime} \, t^{2} + O(t^{4})
\quad \text{as $t\to 0$,}
\end{equation*}
from which Lemma~\ref{Prop:quad_est} follows.
\end{proof}

%%%%%%%%%%%%%%%%%%%%%%%%%%%%%%%%%%%%%%%%
\subsubsection{Proof of Theorem~\ref{Thm:raise}--(1)}
%%%%%%%%%%%%%%%%%%%%%%%%%%%%%%%%%%%%%%%%

Assume that the function $b$ satisfies Condition (A1) and $\Vert b \Vert_{L_{2}(\mathbf{R})} = 1$.
By Proposition~\ref{r:diff'ble} and Taylor's expansion, there exists
$c>0$ and $c^{\prime} \in ( 0, c^{-1/2} )$
such that
\begin{equation}\label{eq:r-ubound}
r(t) \leqslant 1 - c \, t^{2}
\quad \text{for $0 \leqslant t \leqslant c^{\prime}$.}
\end{equation}
Consider a partition
$0 = t_{0} < t_{1} < \cdots < t_{m} = 1$
such that
$\displaystyle \max_{1 \leqslant l \leqslant m} \{ t_{l} - t_{l-1}\} < c^{\prime}$.

The chaos expansion of $\Lambda ( X_{t} )$
is given as 
$$
\Lambda ( X_{t} ) =
\sum_{n=0}^{\infty} J_{n} [ \Lambda ( X_{t} ) ],\quad\text{where}
$$
\begin{equation*}
\begin{split}
J_{n} [ \Lambda ( X_{t} ) ]
&=
n! \, \mathbb{E}
\left[
	\Lambda ( X_{t} ) \int_{x_{1} < \cdots < x_{n}} b (t+x_{1}) \cdots b (t+x_{n}) \,\mathrm{d}w(x_{1}) \cdots \mathrm{d}w(x_{n})
\right] \\
&\hspace{30mm}\times
\int_{x_{1} < \cdots < x_{n}}
b (t+x_{1}) \cdots b (t+x_{n})
\,\mathrm{d}w(x_{1}) \cdots \mathrm{d}w(x_{n}).
\end{split}
\end{equation*}
By stationarity,
$$
\mathbb{E}
\left[
	\Lambda ( X_{t} )
	\int_{x_{1} < \cdots < x_{n}}
	b (t+x_{1}) \cdots b (t+x_{n})
    \,
	\mathrm{d}w(x_{1}) \cdots \mathrm{d}w(x_{n})
\right] := c^{(n)}
$$
is independent of $t$. Therefore, we have
$$
\Vert J_{n} [ \Lambda (X_{0}) ] \Vert_{L_{2}}^{2} = 
( c^{(n)} )^{2}\, n!
$$
and thus,
\begin{equation*}
\begin{split}
\left\Vert
	\int_{t_{l-1}}^{t_{l}}
	J_{n} [ \Lambda (X_{t}) ]
	\,\mathrm{d} t
\right\Vert_{2}^{2}
&= %%%%%%%%%
2 (n!) ( c^{(n)} )^{2}
\int_{ t_{l-1} \leqslant s < u \leqslant t_{l} }
\left(
	\int_{-\infty}^{\infty}
		b(s+x)
		b(u+x)
		\,\mathrm{d}x
\right)^{n}
\mathrm{d}s
\mathrm{d}u \\
&= %%%%%%%%%
2
\Vert J_{n} [ \Lambda ( X_{0} ) ] \Vert_{L_{2}}^{2}
\int_{ 0 \leqslant s < u \leqslant t_{l} - t_{l-1} }
( r (s) )^{n}
\, \mathrm{d}s \mathrm{d}u \\
&\leqslant %%%%%%%%%
2
\Vert J_{n} [ \Lambda ( X_{0} ) ] \Vert_{L_{2}}^{2}
\int_{ 0 \leqslant s < u \leqslant c^{\prime} }
( 1 - c \, s^{2} )^{n}
\, \mathrm{d}s \mathrm{d}u .
\end{split}
\end{equation*}
by using \eqref{eq:r-ubound} in the last inequality.

Finally, applying Lemma~\ref{Lem:iter-asym}
concludes the proof of part (1) of Theorem~\ref{Thm:raise}.

%%%%%%%%%%%%%%%%%%%%%%%%%%%%%%%%%%%%%%%%
\subsubsection{Proof of Theorem~\ref{Thm:raise}--(2)}
%%%%%%%%%%%%%%%%%%%%%%%%%%%%%%%%%%%%%%%%

%\begin{proof}[Proof of Theorem~\ref{Thm:raise}]
Assume that the function $b$ satisfies Conditions (A1), (A2), and that $\Vert b \Vert_{L_{2}(\mathbf{R})} = 1$.
Since
$\langle b, b^{\prime} \rangle_{L_{2 (\mathbf{R})}}
= r^{\prime} (0) = 0$,
we see that
$$
e^{1} := b,\quad
e^{2} := b^{\prime} / \Vert b^{\prime} \Vert_{L_{2}(\mathbf{R})} = b^{\prime}/(- r^{\prime\prime}(0))
$$
forms an orthonormal system.
For each $h \in L_{2}(\mathbf{R})$ and $s \geqslant 0$, we write
$\displaystyle h_{s}(x) := h(s+x)$.
The chaos expansion for
$\displaystyle \Lambda ( X_{s}, \dot{X}_{s} )$
is then
$$
\Lambda ( X_{s}, \dot{X}_{s} ) =
\sum_{n=0}^{\infty} J_{n} [ \Lambda ( X_{s}, \dot{X}_{s} ) ], \quad \text{where}
$$
\begin{equation*}
\begin{split}
J_{n} [ \Lambda ( X_{s}, \dot{X}_{s} ) ]
\,=\,&
n!
\sum_{i_{1}, \ldots , i_{n} \in \{ 1,2 \}}
\mathbb{E}
\left[
	\Lambda ( X_{s}, \dot{X}_{s} )
	\int_{x_{1} < \cdots < x_{n}}
	e_{s}^{i_{1}} (x_{1}) \cdots e_{s}^{i_{n}} (x_{n})
	\,\mathrm{d}w(x_{1}) \cdots \mathrm{d}w(x_{n})
\right] \\
&\hspace{20mm}\times
\int_{x_{1} < \cdots < x_{n}}
e_{s}^{i_{1}} (x_{1}) \cdots e_{s}^{i_{n}} (x_{n})
\,\mathrm{d}w(x_{1}) \cdots \mathrm{d}w(x_{n}).
\end{split}
\end{equation*}
By stationarity, we have, for each $i_{1}, i_{2}, \ldots , i_{n} \in \{ 1, 2 \}$,
\begin{equation*}
\mathbb{E}
\left[
	\Lambda ( X_{s}, \dot{X}_{s} )
	\int_{x_{1} < \cdots < x_{n}}
	\prod_{k=1}^{n} e_{s}^{i_{k}} (x_{k}) 
	\,\prod_{k=1}^{n}\mathrm{d}w(x_{k}) 
\right] \\
=
\mathbb{E}
\left[
	\Lambda ( X_{0}, \dot{X}_{0} )
	\int_{x_{1} < \cdots < x_{n}}
	\prod_{k=1}^{n} e^{i_{k}} (x_{k}) 
	\,\prod_{k=1}^{n} \mathrm{d}w(x_{k}) 
\right].
\end{equation*}
Let $c, c^{\prime} > 0$ be as in
Lemma~\ref{Prop:quad_est}.
Considering a partition
$0 = t_{0} < t_{1} < \cdots < t_{m} = 1$
such that
$\displaystyle
\max_{1\leqslant l \leqslant m} \{ t_{l} - t_{l-1} \}< c^{\prime}$, and denoting 
$$
c_{(i_{1}, i_{2}, \ldots , i_{n})}
:=:c_{i_{1}, i_{2}, \ldots , i_{n}}
:=
\mathbb{E}
\left[
	\Lambda ( X_{0}, \dot{X}_{0} )
	\int_{x_{1} < \cdots < x_{n}}
	e^{i_{1}} (x_{1}) \cdots e^{i_{n}} (x_{n})
	\,\mathrm{d}w(x_{1}) \cdots \mathrm{d}w(x_{n})
\right],
$$
we can write, for each $l=1,2,\ldots , m$,
\begin{equation*}
\begin{split}
&
\left\Vert
\int_{t_{l-1}}^{t_{l}}
J_{n} [ \Lambda ( X_{s}, \dot{X}_{s} ) ]
\,\mathrm{d}s
\right\Vert_{2}^{2} \,=\,
n!
\sum_{
	\substack{
		i_{1}, \ldots , i_{n} \in \{ 1,2 \}, \\
		j_{1}, \ldots , j_{n} \in \{ 1,2 \}
	}
}
c_{ i_{1}, i_{2}, \ldots , i_{n} }
c_{ j_{1}, j_{2}, \ldots , j_{n} }\; \times\\
&
\mathbb{E}
\left[
\left\{
\int_{t_{l-1}}^{t_{l}} \!\!\mathrm{d}s
\int_{x_{1} < \cdots < x_{n}}
\prod_{k=1}^n e_{s}^{i_{k}} (x_{k}) 
\,\prod_{k=1}^n\mathrm{d}w(x_{k})
\right\}
\left\{
\int_{t_{l-1}}^{t_{l}} \!\! \mathrm{d}s
\int_{y_{1} < \cdots < y_{n}}
\prod_{k=1}^n e_{s}^{j_{k}} (y_{k}) 
\,\prod_{k=1}^n \mathrm{d}w(y_{k}) 
\right\}
\right] \\
&= %%%%%%%%%
2 ( n! )
\sum_{
	\substack{
		i_{1}, \ldots , i_{n} \in \{ 1,2 \}, \\
		j_{1}, \ldots , j_{n} \in \{ 1,2 \}
	}
}
c_{ i_{1}, i_{2}, \ldots , i_{n} }
c_{ j_{1}, j_{2}, \ldots , j_{n} }
\int_{t_{l-1} \leqslant s < u \leqslant t_{l}}
\left\{
\prod_{k=1}^{n}
\int_{-\infty}^{+\infty}
e_{s}^{i_{k}} (x)
e_{u}^{j_{k}} (x)
\,\mathrm{d}x
\right\}
\mathrm{d}s\,\mathrm{d}u .
\end{split}
\end{equation*}

Using the matrix $A$ defined in \eqref{def:matrixA}, we can write 
\begin{equation*}
A_{ij}(u-s)
= \int_{-\infty}^{+\infty} e_{s}^{i} (x) e_{u}^{j} (x)
\,\mathrm{d}x \quad \text{for $i,j \in \{ 1, 2 \}$,}
\end{equation*}
and we define, for each $n \in \mathbf{N}$ and
$
I=(i_{1}, \ldots , i_{n}), J=(j_{1}, \ldots , j_{n})
\in \{ 1, 2 \}^{n}
$,
\begin{equation*}
A_{I,J}^{(n)} (u-s)
:= \prod_{k=1}^{n} \int_{-\infty}^{+\infty} e_{s}^{i_{k}} (x) e_{u}^{j_{k}} (x)\,\mathrm{d}x
= \prod_{k=1}^{n} A_{i_{k},j_{k}} (u-s) ,
\quad \text{for $s < u$.}
\end{equation*}
We also write
$$
A^{(n)} (s) := ( A_{I,J}^{(n)} (s) )_{ I, J \in \{ 1, 2 \}^{n} } 
$$
and note that $A^{(1)} (s) \equiv A(s)$.
Then, after a change of variables, we obtain
\begin{equation*}
\begin{split}
\left\Vert \int_{t_{l-1}}^{t_{l}} J_{n} [ \Lambda ( X_{s}, \dot{X}_{s} ) ] \,\mathrm{d}s \right\Vert_{2}^{2}
= %%%%%%%%%
2 ( n! )
\int_{ 0 \leqslant s < u \leqslant t_{l} - t_{l-1} }
\langle \mathbf{c}^{(n)}, A^{(n)} ( s ) \mathbf{c}^{(n)} \rangle_{\mathbf{R}^{2^{n}}}
\, \mathrm{d}s \,\mathrm{d}u,
\end{split}
\end{equation*}
where
$\displaystyle  \mathbf{c}^{(n)} :=
( c_{I} )_{ I \in \{ 1, 2 \}^{n} }$
and
$\displaystyle \langle \cdot , \cdot \rangle_{\mathbf{R}^{2n}}$
denotes the canonical inner product on 
 $\mathbf{R}^{2n}$.

Under the canonical basis on $\mathbf{R}^{2^{n}}$,
$\{ \mathbf{e}_{I} \}_{I \in \{ 1, 2 \}^{n}}$,
defined recursively as
\begin{equation*}
\begin{split}
\mathbf{e}_{1} &:= \left(\begin{array}{c} 1 \\ 0 \end{array}\right),
\quad
\mathbf{e}_{2} := \left(\begin{array}{c} 0 \\ 1 \end{array}\right), \\
%%%%%%%%%%%%
\mathbf{e}_{11}
&:=
\left(\begin{array}{c}
\mathbf{e}_{1} \\\hdashline
\mathbf{0}_{\mathbf{R}^{2}}
\end{array}\right),
\quad
\mathbf{e}_{12}
:=
\left(\begin{array}{c}
	\mathbf{e}_{2} \\\hdashline
	\mathbf{0}_{\mathbf{R}^{2}}
\end{array}\right),
\quad
\mathbf{e}_{21}
:=
\left(\begin{array}{c}
	\mathbf{0}_{\mathbf{R}^{2}} \\\hdashline
	\mathbf{e}_{1}
\end{array}\right),
\quad
\mathbf{e}_{22}
:=
\left(\begin{array}{c}
	\mathbf{0}_{\mathbf{R}^{2}} \\\hdashline
	\mathbf{e}_{2}
\end{array}\right), \\
%%%%%%%%%%%%
& \vdots \\
%%%%%%%%%%%%
\mathbf{e}_{(1,I)}
&:=
\left(\begin{array}{c}
	\mathbf{e}_{I} \\\hdashline
	\mathbf{0}_{\mathbf{R}^{2^{n-1}}}
\end{array}\right) ,
\quad
\mathbf{e}_{(2,I)}
:=
\left(\begin{array}{c}
	\mathbf{0}_{\mathbf{R}^{2^{n-1}}} \\\hdashline
	\mathbf{e}_{I}
\end{array}\right)
\quad
\text{for $I \in \{ 1,2 \}^{n-1}$},
\end{split}
\end{equation*}
where $\mathbf{0}_{\mathbf{R}^{k}}$ denotes the zero vector in $\mathbf{R}^{k}$, 
the matrices $A^{(n)} = A^{(n)} (t)$, $n=1,2,3,\ldots$, are related as
\begin{equation*}
A^{(n)}
=
\left(\begin{array}{c:c}
A_{11} A^{(n-1)} & A_{12} A^{(n-1)} \\ \hdashline
A_{21} A^{(n-1)} & A_{22} A^{(n-1)}
\end{array}\right)
= A \otimes A^{(n-1)}
= A^{\otimes n}.
\end{equation*}
Thus, by Lemma~\ref{Prop:quad_est}, we have
\begin{equation*}
\Vert A^{(n)} (s) \Vert_{\mathrm{HS}}
= \Vert A (s) \Vert_{\mathrm{HS}}^{n} \leqslant
( 1 - c \, s^{2} )^{n}
\quad
\text{for $0 \leqslant s \leqslant c^{\prime}$, $n \in \mathbf{N}$}.
\end{equation*}
Now, by using that
\begin{equation*}
(n!)
\Vert \mathbf{c}^{(n)} \Vert_{\mathbf{R}^{2^{n}}}^{2}
=
\Vert J_{n} [ \Lambda ( X_{0}, \dot{X}_{0} ) ] \Vert_{L_{2}}^{2},
\end{equation*}
and the Cauchy--Schwartz inequality, we obtain 
\begin{equation*}
\left\Vert
	\int_{t_{l-1}}^{t_{l}}
		J_{n} [
        \Lambda ( X_{s}, \dot{X}_{s} )
        ]
	\,\mathrm{d}s
\right\Vert_{L_{2}}^{2}
\,\leqslant\, %%%%%%%%%
2\, \Vert J_{n} [ \Lambda ( X_{0}, \dot{X}_{0} ) ] \Vert_{L_{2}}^{2}
\int_{ 0 \leqslant s < u \leqslant t_{l} - t_{l-1} }
( 1 - c \, s^{2} )^{n}
\,\mathrm{d}s \,\mathrm{d}u,
\end{equation*}
from which we deduce, by Lemma~\ref{Lem:iter-asym}, that there exists a constant $c^{\prime\prime}>0$ such that
\begin{equation*}
\left\Vert
\int_{t_{l-1}}^{t_{l}} J_{n} [ \Lambda ( X_{s}, \dot{X}_{s} ) ] \,\mathrm{d}s
\right\Vert_{2}^{2}
\!\leqslant\!
c^{\prime\prime}(n+1)^{-1/2}
\left\Vert
	J_{n} [ \Lambda ( X_{0}, \dot{X}_{0} ) ]
\right\Vert_{L_{2}}^{2}
\end{equation*}
for any $n \in \mathbf{Z}^+$.
Therefore, we can write
\begin{equation*}
\begin{split}
\Vert \int_{t_{l-1}}^{t_{l}}
\Lambda ( X_{s}, \dot{X}_{s} )
&\,\mathrm{d}s 
\Vert_{2, \alpha + \frac{1}{2}}^{2}
= \sum_{n=1}^{\infty}
(1+n)^{ \alpha + \frac{1}{2} }
\Vert
\int_{t_{l-1}}^{t_{l}}
J_{n} [ \Lambda ( X_{s}, \dot{X}_{s} ) ]
\,\mathrm{d}s
\Vert_{2}^{2} \\
&\leqslant c^{\prime\prime} \sum_{n=1}^{\infty} (1+n)^{ \alpha + \frac{1}{2} } (1+n)^{-\frac{1}{2}}
\Vert J_{n} [ \Lambda ( X_{0}, \dot{X}_{0} ) ] \Vert_{2}^{2}
\leqslant c^{\prime\prime} \Vert \Lambda ( X_{0}, \dot{X}_{0} ) \Vert_{ 2, \alpha }^{2} .
\end{split}
\end{equation*}
and then,
\begin{equation*}
\left\Vert
\int_{0}^{1} \Lambda ( X_{s}, \dot{X}_{s} ) \,\mathrm{d}s
\right\Vert_{ 2, \alpha + \frac{1}{2} }
\leqslant
\sum_{l=1}^{m} \left \Vert \int_{t_{l-1}}^{t_{l}} \Lambda ( X_{s}, \dot{X}_{s} ) \,\mathrm{d}s
\right\Vert_{ 2, \alpha + \frac{1}{2} }
\leqslant (c^{\prime\prime})^{1/2} \,m\,
\Vert \Lambda ( X_{0}, \dot{X}_{0} ) \Vert_{ 2, \alpha }. \qquad \Box
\end{equation*}
% \end{proof} %%%%%%%%%%%%%%%%%%%%%%%%%%%%

%%%%%%%%%%%%%%%%%%%%%%%%%%%%%%%%%%%%%%%%
\subsection{Proof of Corollary~\ref{Cor:Ocup_Sob}}
%%%%%%%%%%%%%%%%%%%%%%%%%%%%%%%%%%%%%%%%

We will recall existing results and derive new ones, which will be useful for our proof and provide context.
\begin{Thm}[Watanabe \cite{Wa91}] %%%%%
\label{Donsker-Delta} %@@@@@@@@@@@@@@@@@ LABEL: Donsker-Delta
Let $\alpha > 0$, $p \in (1, \infty)$ and $ F \in \mathbb{D}^{\infty} (\mathbf{R}) $ be non-degenerate in the sense of Malliavin.
Then,
\begin{equation*}
\begin{split}
\delta_{x} (F) \in \mathbb{D}_{p}^{-2\alpha}
\Longleftrightarrow
\left\{\begin{array}{l}
\text{{\rm (i) $\alpha \geqslant 1/2$ and $p \in (1,\infty )$,}} \\
\text{{\rm or}} \\
\text{{\rm (ii) $0 < \alpha < 1/2$ and $1 < p < \frac{1}{1-2\alpha}$.}}
\end{array}\right.
\end{split}
\end{equation*}
\end{Thm} %%%%%%%%%%%%%%%%%%%%%%%%%%%%%

It is worth noting that this result has been refined by Nualart and Vives %, who established it 
for $F$ as a Wiener integral.
\begin{Thm}[{Nualart--Vives~\cite[Theorem 2.1]{NVives92b}}] %%%%%
\label{Donsker-Delta-BM} %@@@@@@@@@@@@@@@@@ LABEL: Donsker-Delta-BM
Let $x\in \mathbf{R}$ and $W(h)$ a Wiener integral ($h\neq 0$). Then, 
\begin{equation*}
\delta_{x} (W(h)) \in \mathbb{D}_{p}^{-\beta}
\quad
\text{if} \quad \beta + \frac1p >1,\, \beta>0, \,p>1
\end{equation*}
and
\begin{equation*}
\delta_{x} (W(h)) \notin \mathbb{D}_{2}^{-1/2}. 
\end{equation*}
\end{Thm} %%%%%%%%%%%%%%%%%%%%%%%%%%%%%

%%%%%%%%%%%%%%%%%%%%%%%%%%%%%%%%%%%%%%%%
\begin{Prop}[Watanabe \cite{Wa93}]
\label{product_est} %@@@@@@@@@@@@@@@@@@@ LABEL: product_est
For each $\alpha \geqslant 0$ and $1<p,q,r < \infty$ such that \\$ \frac1p + \frac1q = \frac1r$,
there exists a constant $c= c(p,q,\alpha )$ such that
\begin{equation*}
\Vert FG \Vert_{r,\alpha} \leqslant c\,  \Vert F \Vert_{p,\alpha} \Vert G \Vert_{q,\alpha}
\quad
\text{for all $(F,G) \in \mathbb{D}_{p}^{\alpha} \times \mathbb{D}_{q}^{\alpha}$.}
\end{equation*}
\end{Prop} %%%%%%%%%%%%%%%%%%%%%%

From Proposition~\ref{product_est} and duality, we obtain the following result.
\begin{Lem} %%%%%%%%%%%%%%%%%%%%
\label{product_est2} %@@@@@@@@@@@@@@@@@@ LABEL: product_est2
For each $\alpha \geqslant 0$ and $1<p,q,r < \infty$ such that $\frac1p + \frac1q = \frac1r$,
there exists a constant $c= c(p,q,\alpha)$ such that
\begin{equation*}
\Vert FG \Vert_{r,-\alpha} \leqslant c\, \Vert F \Vert_{p,\alpha} \Vert G \Vert_{q,-\alpha}
\quad
\text{for all $(F,G) \in \mathbb{D}_{p}^{\alpha} \times \mathbb{D}_{q}^{-\alpha}$.}
\end{equation*}
\end{Lem} %%%%%%%%%%%%%%%%%%%%%%
\begin{proof} %%%%%%%%%%%%%%%%%%%%%%%%%%
Suppose that $p,q,r$ as above. Let $r^{\prime} \in (1,\infty )$ be such that $1/r + 1/r^{\prime} = 1$.
It suffices to prove that, for any $J \in \mathbb{D}_{r^{\prime}}^{\alpha}$,
\begin{equation*}
\vert \mathbb{E}[ (FG)J ] \vert \leqslant
c \, \Vert F \Vert_{p,\alpha} \Vert G \Vert_{q,-\alpha}
\, \Vert J \Vert_{r^{\prime},\alpha},
\end{equation*}
with $c$ a positive constant.

In fact, we can write
$\displaystyle
\vert \mathbb{E}[ (FG)J ] \vert \leqslant \Vert FJ \Vert_{q^{\prime},\alpha} \Vert G \Vert_{q,-\alpha}
$,
where $1/q + 1/q^{\prime} = 1$.
Since we have
$
\frac{1}{q^{\prime}} = 1 - \frac{1}{q} = 1 - ( \frac{1}{r} - \frac{1}{p} ) 
= ( 1 - \frac{1}{r} ) + \frac{1}{p} = \frac{1}{r^{\prime}} + \frac{1}{p}
$, 
applying Proposition~\ref{product_est} provides
$\displaystyle
\Vert FJ \Vert_{q^{\prime},\alpha} \leqslant c\,  \Vert F \Vert_{p,\alpha} \Vert J \Vert_{r^{\prime},\alpha}
$,
from which the claim follows.
\end{proof} %%%%%%%%%%%%%%%%%%%%%%%%%%%%

While Nualart and Vives (see \cite{NVives92b}), as well as Imkeller et al. (see \cite{im:pv}), focused on the smoothness of Brownian local time, we instead examine the smoothness of the number of crossings of our stationary Gaussian process, applying our general method.

\begin{Prop} %%%%%%%%%%%%%%%%%%%%
\label{spec_class} %@@@@@@@@@@@@@@@@@@@@ LABEL: spec_class
For any $p \in (1,\infty )$, we have
\begin{itemize}
\item[(i)] $ \displaystyle \vert \dot{X_{t}} \vert
\in \cap_{p \in (1,\infty )} \mathbb{D}_{p}^{1} $,
\item[(ii)] $\displaystyle
\delta_{x} (X_{t}) \vert \dot{X}_{t} \vert \in \mathbb{D}_{p}^{(\frac{1}{p}-1)-} $.
\end{itemize}
\end{Prop} %%%%%%%%%%%%%%%%%%%%%%

\begin{proof} %%%%%%%%%%%%%%%%%%%%%%%%%%
\hspace{-2ex}\begin{itemize}
    \item[(i)] We have
$\displaystyle
D \vert \dot{X_{t}} \vert
= \mathrm{sgn} ( \dot{X_{t}} ) D \dot{X_{t}}
= \mathrm{sgn} ( \dot{X_{t}} ) \dot{b_{t}}
\in \cap_{p \in (1,+\infty )} L_{p} (H)
$,
from which result (i) follows.
Notice that, when $p=2$, result (i) reduces to a specific case of Proposition 1.1.5 in \cite{MK-HDR2005}, which states that 
$\displaystyle \vert \dot{X_{t}} \vert \in \mathbb{D}_{2}^{\alpha}$, for any $\alpha$.
    
    \item[(ii)]
    Take $\alpha \in (0,1)$ and $r \in (1, \frac{1}{1-\alpha})$ arbitrarily.
Then, choose $q \in (r, \frac{1}{1-\alpha})$
and define $p \in (1,\infty )$ by
$\displaystyle \frac{1}{p} = \frac{1}{r} - \frac{1}{q}$.
Next, applying Lemma~\ref{product_est2} yields
\begin{equation*}
\left\Vert \delta_{x} (X_{t}) \vert \dot{X}_{t} \vert \right\Vert_{r,-\alpha}
\leqslant c\,
\Vert \delta_{x} (X_{t}) \Vert_{q,-\alpha}
\big\Vert \vert \dot{X}_{t} \vert \big\Vert_{p,\alpha}
\leqslant c\,
\Vert \delta_{x} (X_{t}) \Vert_{q,-\alpha}
\big\Vert \vert \dot{X}_{t} \vert \big\Vert_{p,1}
< +\infty,
\end{equation*}
which establishes result (ii).
\end{itemize}
\vspace{-3ex}
\end{proof} %%%%%%%%%%%%%%%%%%%%%%%%%%%%

Let us recall one final result, needed to complete the proof of Corollary~\ref{Cor:Ocup_Sob}.

%%%%%%%%%%%%%%%%%%%%%%%%%%%%%%%%%%%%%%%%
\begin{Prop}[{\cite[Lemma~4.10]{amaba:ryu}}]
\label{frac_ineq} %@@@@@@@@@@@@@@@@@@@@@ LABEL: frac_ineq
For any $\alpha \in \mathbf{R}$,
$p \in (1,\infty )$
and
$p^{\prime} > p$,
there exists
$c = c(s,p,p^{\prime}) > 0$
such that
\begin{equation*}
\Vert \Lambda (\dot{X}_{t}) \Vert_{p, \alpha}
\leqslant
c \Vert \Lambda \Vert_{ H_{p^{\prime}}^{\alpha} (\mathbf{R}) }
\end{equation*}
for any $t \in \mathbf{R}$
and
$\Lambda \in H_{p^{\prime}}^{\alpha} (\mathbf{R})$, where $H_{p^{\prime}}^{\alpha} (\mathbf{R})$ denotes a Bessel potential space (defined in Section~\ref{Sec:Results}).
\end{Prop} %%%%%%%%%%%%%%%%%%%%%%

%%%%%%%%%%%%%%%%%%%%%%%%%%%%%%%%%%%%%%%%
%\begin{proof}[Proof of Corollary~\ref{Cor:Ocup_Sob}]

Now, let $p \in (1,\infty)$, $\frac{1}{2} < \alpha \leqslant 1$, and
$f \in H_{p}^{\alpha} (\mathbf{R})$.
Let $ r \in ( \frac{p}{p+1}, \frac{p}{1+p(1-\alpha )} )$ be arbitrary.
Take $p^{\prime} \in ( p, \infty )$ such that 
$ \frac{ p^{\prime} }{ p^{\prime} + 1 } < r $.
Then, it holds that
$ \frac{ p^{\prime} }{ p^{\prime} + 1 } < r < \frac{ p^{\prime} }{ 1 + p^{\prime} ( 1-\alpha ) }
$.
Let us define $q \in (1, \infty )$ by
$\displaystyle 
\frac{1}{q} = \frac{1}{r} - \frac{1}{p^{\prime}}
$;
so, we have $1 < q < \frac{1}{1-\alpha}$.
Now, by using Lemma~\ref{product_est2} and Proposition~\ref{frac_ineq}, with $c$ a positive constant that may change from line to line, we obtain
\begin{equation*}
\begin{split}
\Vert \delta_{x} ( X_{t} ) f ( \dot{X}_{t} ) \Vert_{r, -\alpha} 
\leqslant c \Vert \delta_{x} ( X_{t} ) \Vert_{q, -\alpha} \Vert f ( \dot{X}_{t} ) \Vert_{p^{\prime}, \alpha}
\leqslant c \Vert \delta_{x} ( X_{t} ) \Vert_{q, -\alpha} \Vert f \Vert_{ H_{p}^{\alpha} (\mathbf{R}) }.
\end{split}
\end{equation*}
%for some constants $c,c^{\prime} > 0$.
By Theorem~\ref{Donsker-Delta}, the last quantity is finite.

Note that $r$ can be chosen as $r=2$; hence, we obtain
\begin{equation*}
\delta_{x} ( X_{t} ) f ( \dot{X}_{t} ) \in \mathbb{D}_{2}^{(-\frac{1}{2})-},
\end{equation*}
from which, and by combining with Theorem~\ref{Thm:raise}, we deduce the result.
\hfill $\Box$
%\end{proof} %%%%%%%%%%%%%%%%%%%%%%%%%%%%

\vspace{2ex}
{\bf Acknowledgments.}
This work benefited greatly from mutual visits to France, Switzerland, and Japan; the authors are grateful to the hosting institutions.
T.~A. was partially supported by funding from Fukuoka University (Grant No. 197102) and by JSPS KAKENHI Grant Number 22K03345.

%%%%%%%%%%%%%%%%%%%%%%%%%%%%%%%%%%%%%%%%
%%%%%%%%%%%% BIBLIOGRAPHY %%%%%%%%%%%%%%
%%%%%%%%%%%%%%%%%%%%%%%%%%%%%%%%%%%%%%%%

\end{document}